\documentclass[a4paper,article]{lms}
\usepackage[english]{babel}
\usepackage{amsmath,amssymb,mathtools,bm}
\usepackage{todonotes}
\usepackage{graphicx}

\newtheorem{theorem}{Theorem}[section] 
\newtheorem{lemma}[theorem]{Lemma}
\newtheorem{corollary}[theorem]{Corollary}
\newtheorem{proposition}[theorem]{Proposition}
\newnumbered{assertion}{Assertion} 
\newnumbered{conjecture}{Conjecture}
\newnumbered{definition}{Definition}
\newnumbered{hypothesis}{Hypothesis}
\newnumbered{remark}{Remark}
\newnumbered{note}{Note}
\newnumbered{observation}{Observation}
\newnumbered{problem}{Problem}
\newnumbered{question}{Question}
\newnumbered{algorithm}{Algorithm}
\newnumbered{example}{Example}
\newunnumbered{notation}{Notation}

\newcommand{\Fbarq}{{\overline{\mathbb F}_q}}
\newcommand{\Fq}{{\mathbb F_q}}
\newcommand{\Fqn}{{\mathbb F_q^n}}
\newcommand{\Fqk}{{\mathbb F_q^k}}
\newcommand{\Z}{{\mathbb Z}}
\newcommand{\F}{\mathbb F_q^*}

   \newcommand{\Spec}{\operatorname{Spec}}
   
   \newcommand{\Hom}{\operatorname{Hom}}

\newcommand{\orb}{\operatorname{orb}}

    \newcommand{\ket}[1]{\lvert#1\rangle}
 \DeclareMathOperator{\tr}{tr}
\DeclareMathOperator{\wt}{wt}
  
\usepackage{xcolor,xparse,tikz,ragged2e}

\usepackage[
final
]{fixme}
\fxsetup{
  multiuser,
  marginface=\normalfont\tiny,
  innerlayout=noinline,
  layout=marginnote,
}

\AtEndDocument{ \clearpage \listoffixmes }

\makeatother

\DeclarePairedDelimiterX\blp[2]\langle\rangle{
  \color{cyan}
  \ifblank{#1}{{-}}{#1}, \ifblank{#2}{{-}}{#2}
  \normalcolor
}

\title[Toric surfaces, Quantum Codes and Secret Sharing]{Toric Surfaces, Linear and  Quantum Codes\\ Secret Sharing and Decoding}

\author{Johan P. Hansen}
\dedication{In genuine gratefulness and with warmest regard I dedicate this work to W. Fulton.

For more than forty years Fulton's knowledge, insight  and unrelenting creativity has served as an inspiration for me.}
\classno{14M25 (primary), 81P70, 94A62, 94B27, 94B35 (secondary)}
\begin{document}
\maketitle

\begin{abstract}

Toric varieties and  their associated toric codes, as well as determination of their parameters with intersection theory,  are presented in the two dimensional case.

Linear Secret Sharing Schemes with strong multiplication  are constructed from toric varieties and codes by the J. L. Massey construction. 

Asymmetric Quantum Codes are obtained from toric codes by the A.R. Calderbank 
P.W. Shor and 
A.M. Steane construction of stabilizer codes  from linear codes containing their dual codes.

Decoding of a class of toric codes is presented.

\end{abstract}

\tableofcontents

\section{Introduction}

We present error-correcting codes obtained from toric varieties of dimension $n$ resulting in long codes of length $q^n$ over the finite  ground field $\mathbb F_q$. In particular we apply the construction to certain toric surfaces. The code parameters are estimated using intersection theory.

Linear secret sharing schemes from
error-correcting codes can be constructed  by the method of J. L. Massey. We  apply his method to two classes of toric  codes giving ideal schemes with strong multiplication with respect to certain adversary structures.

Quantum error-correcting codes can be constructed by the method of A. R. Calderbank, P. W. Shor and A. M. Steane -- the (CSS) method. Our construction  of toric codes is suitable for making
quantum codes and
extends similar results obtained by A. Ashikhmin, S. Litsyn and
M.A. Tsfasman  obtained from Goppa codes on algebraic curves.

We utilized the inherent multiplicative structure on toric codes
to decode a class of toric codes.

\subsection{Error correcting codes}

Codes are used in communication and storage of information.

The message is divided into blocks and extra information is appended
before transmission allowing the receiver to correct errors.

Let $\Fq$ be a finite field with $q$ elements.  A word of length $n$ in the
alphabet $\Fq$ is a vector
\begin{equation*}
  {\bf c}=(c_1, c_2, \dots, c_n) \in \Fqn.
\end{equation*}
The Hamming weight $w({\bf c})$ is the number of non-zero coordinates in ${\bf c}$.
The Hamming distance $d({\bf c}_1,{\bf c}_2)$ between two words is the
Hamming weight $w({\bf c}_1-{\bf c}_2)$.

A linear code is a $\Fq$-linear subspace $C \subseteq \Fqn$.  The dimension of the code is $k=\dim_{\Fq} C$ and the minimum
distance $d(C)$ of the code $C$ is the minimal Hamming distance
$d({\bf c}_1,{\bf c}_2)=w({\bf c}_1-{\bf c}_2)$ between two code
words ${\bf c}_1, {\bf c}_2 \in C$ with ${\bf c}_1 \neq {\bf c}_2$.

A linear code $C$ can correct errors in $t$ coordinates or less, if and only if
$t < \frac{d(C)}{2}$.

For general presentations of the theory of error-correcting codes, see  \cite{MR0465510} and \cite{tom}.

\begin{example}[(Reed-Solomon code)]\label{def:RS} Let
  $x_1, x_2, \dots, x_n \in \Fq$ be $n$ distinct elements and let
  $0<k \leq n$.

  To the word $(a_0, a_1, \dots, a_{k-1}) \in \Fqk$ of length $k$ we
  associate the polynomial
  \begin{equation*}
    f(X)=a_0+a_1 X+ \dots + a_{k-1}X^{k-1} \in \Fq[X]
  \end{equation*}
  and upon evaluation the Reed--Solomon code word
  \begin{equation*}
    \bigl(f(x_1), f(x_2),\dots, f(x_n)\bigr) \in \Fqn\ .
  \end{equation*}
  The Reed--Solomon code $C_{n.k} \subseteq \Fqn$ is the subspace of
  all Reed--Solomon code words coming from all polynomials $f(X) \in \Fq[X]$ with
  $\deg f(X)<k \leq n$.

  The Reed--Solomon code $C_{n.k} \subseteq \Fqn$ has dimension $\dim_{\Fq}(C_{n,k})= k$,
  minimum distance $d(C_{n.k})=n-k+1$ and corrects
  $t <\frac{d(C_{n.k})}{2}$ errors.
\end{example}

\begin{example}[(Goppa code)]\label{def:Goppa}
Let $X$ be a projective algebraic curve defined over the finite field $\Fq$. Let $D$ be a $\Fq$-rational divisor on $X$, and let $L(D)$ be the $\Fq$-rational functions  $f$ on $X$ with $\mathrm{div}  f  + D \geq 0$. Let $P_1, P_2, \dots P_n \in X(\Fq)$ be $\Fq$-rational points on $X$, none in the support of $D$.

A $\Fq$--rational function $f \in L(D)$ gives rise to a code word
  \begin{equation*}
    \bigl(f(P_1), f(P_2),\dots, f(P_n)\bigr) \in \Fqn\ .
  \end{equation*}
  The Goppa code $C \subseteq \Fqn$ is the subspace of
  all Goppa code words coming from all $f \in L(D)$, see \cite{MR1029027}.

Good codes are constructed from curves with a large numbers of rational points compared to their genus. Using Shimura curves (modular curves) one find a sequence of codes which beats the Gilbert–Varshamov bound, see \cite{MR705893}. In \cite{MR2464941} curves are presented from the function field point of view and in
\cite{MR3224833} good curves are constructed via Galois towers of function fields. Some Deligne-Luztig curves with large automorphism groups have the maximal number of rational points compared to their genus, see  \cite{MR1186416}, \cite{MR1325513} and
\cite{MR1225959}.
 \end{example}

\section{Toric varieties}

The toric codes presented in Section \ref{sec:toriccodes} are obtained by evaluating certain $\Fq$-rational functions in $\Fq$-rational points on toric varieties defined over the finite field $\Fq$.

For the general
theory of toric varieties we refer to \cite{Fulton:1436535},  \cite{1223.14001},
 and \cite{Oda1988}. Here we will recollect some of their theory needed for determination of the parameters of the toric codes.

\subsection{Polytopes, normal fans and support
  functions}\label{sec:basic}

Let $M \simeq \Z^r$ be a free $\Z$-module of rank r over the integers
$\Z$.  Let $\square$ be an integral convex polytope in
$M_{\mathbb R}= M \otimes_{\mathbb Z}{\mathbb R}$, i.e. a compact
convex polyhedron such that the vertices belong to $M$.

Let $N=\Hom_\Z(M,\Z)$ be the dual lattice with canonical $\Z$ --
bi-linear pairing
\begin{equation*}
 \langle-,-\rangle: M \times N \rightarrow \Z.
\end{equation*}
Let $M_{\mathbb R}= M \otimes_{\mathbb Z}{\mathbb R}$ and
$N_{\mathbb R}= N\otimes_{\mathbb Z}{\mathbb R}$ with canonical
$\mathbb R$ - bi-linear pairing
\begin{equation*}
 \langle-,-\rangle:  M_{\mathbb R} \times N_{\mathbb
    R} \rightarrow  \mathbb R .
\end{equation*}

The $r$-dimensional \emph{algebraic torus}
$T_N \simeq (\Fbarq^*)^r $ is defined by
$T_N:= \Hom_\Z(M,\Fbarq^*)$. The multiplicative character
${\bf e}(m),\, m \in M$ is the homomorphism
\begin{equation*}
  {\bf e}(m): T \rightarrow \Fbarq^*
\end{equation*}
defined by ${\bf e}(m)(t) = t(m)$ for $t \in T_N$.  Specifically, if
$\{n_1,\dots ,n_r\}$ and $\{m_1,\dots, m_r\}$ are dual $\Z$-bases of
$N$ and $M$ and we denote $u_j:= {\bf e}(m_j),\,j=1,\dots, r$, then we
have an isomorphism $T_N \backsimeq (\Fbarq^* )^r$ sending $t$ to
$(u_1(t),\dots, u_r(t))$.  For $m=\lambda_1 m_1 +\dots +\lambda_r m_r$
we have
\begin{equation}\label{char}
  {\bf e}(m)(t)=u_1(t)^{\lambda_1}\cdot \dots \cdot u_r(t)^{\lambda_r}.
\end{equation}

Given an $r$-dimensional integral convex polytope $\square$ in
$M_{\mathbb R}$. The support function
$ h_{\square}: N_{\mathbb R} \rightarrow \mathbb R $ is defined as
$ h_{\square}(n):= \inf \{ <m,n> \vert \  m \in \square \} $ and the polytope
$\square$ can be reconstructed from the support function
\begin{equation}
  \square= \square_{h} =\{ m \in M \vert\ <m,n> \  \geq \ h_{\square}(n) \quad \forall n \in N \}.
  \label{support}
\end{equation}

The support function $h_{\square}$ is piecewise linear in the sense
that $N_{\mathbb R}$ is the union of a non-empty finite collection of
strongly convex polyhedral cones in $N_{\mathbb R}$ such that
$h_{\square}$ is linear on each cone.

A fan is a collection $\Delta$ of strongly convex polyhedral cones in
$N_{\mathbb R}$ such that every face of $\sigma \in \Delta $ is
contained in $\Delta $ and $\sigma \cap \sigma' \in \Delta$ for all
$\sigma , \sigma' \in \Delta $.

The \emph{normal fan} $\Delta$ is the coarsest fan such that
$h_{\square}$ is linear on each $\sigma \in \Delta$, i.e. for all
$\sigma \in \Delta$ there exists $l_{\sigma} \in M$ such that
\begin{equation}
  h_{\square}(n) = <l_{\sigma},n> \quad \forall n \in \sigma.
 \end{equation}

The 1-dimensional cones $\rho \in \Delta$ are generated by unique
primitive elements $n(\rho) \in N \cap \rho$ such that
$\rho =\mathbb R_{\geq 0} \, n(\rho)$.

In the two dimensional case we can upon refinement of the normal fan assume that two successive pairs of $n(\rho)$'s generate the lattice and we obtain \emph{the refined normal fan}.

The \emph{toric surface} $X_{\square}$ associated to the refined
normal fan $\Delta$ of $\square$ is
\begin{equation}\label{surface}
  X_{\square} = \bigcup_{\sigma \in \Delta} U_{\sigma}
\end{equation}
where $U_{\sigma}$ is the $\Fbarq$-valued points of the affine scheme
$\Spec(\Fbarq[S_{\sigma}])$, i.e.
\begin{equation*}
  U_{\sigma}=\{u : S_{\sigma}
    \rightarrow \Fbarq \vert \   u(0)=1,
    \,u(m+m')= u(m)u(m')\quad \forall m,m' \in  S_{\sigma} \} ,
\end{equation*}
where $S_{\sigma}$ is the additive sub-semi-group of $M$
\begin{equation*}
  S_{\sigma}=\{ m \in M \vert \ <m,y> \ \geq \ 0 \quad \forall y \in \sigma \}\ .
\end{equation*}
The \emph{toric surface} $X_{\square}$ of (\ref{surface}) is irreducible, non-singular
and complete  as it is constructed from the refined normal fan.

 If
$\sigma, \tau \in \Delta$ and $\tau$ is a face of $\sigma$, then
$U_{\tau}$ is an open subset of $U_{\sigma}$.  Obviously,
$S_{0}=M$ and $U_{0}=T_N$ such that the algebraic torus $T_N$ is
an open subset of $X_{\square}$.

$T_N$ \emph{acts algebraically} on $X_{\square}$. On
$u \in U_{\sigma}$ the action of $t \in T_N$ is obtained as
\begin{equation*}
  (tu)(m):=t(m)u(m) \qquad  m \in S_{\sigma}
\end{equation*}
such that $tu \in U_{\sigma}$ and $U_{\sigma}$ is $T_N$-stable.  The
orbits of this action is in one-to-one correspondence with
$\Delta$. For each $\sigma \in \Delta$ let
\begin{equation*}
  \orb(\sigma):=\{u:M\cap \sigma \rightarrow \Fbarq^* \vert \  \text{$u$
    is a group homomorphism} \}.
\end{equation*}
Then $\orb(\sigma)$ is a $T_N$ orbit in $X_{\square}$. Define
$V(\sigma)$ to be the closure of $\orb(\sigma)$ in $X_{\square}$.

A $\Delta$-linear support function $h$ gives rise to the Cartier
divisor $D_{h}$. Let $\Delta (1)$ be the 1-dimensional cones in
$\Delta$, then
\begin{equation*}
  D_h:= -\sum_{\rho \in \Delta (1)} h(n(\rho))\,V(\rho).
\end{equation*}
In particular
\begin{equation*}
  D_m=\mathrm{div}({\bf e}(-m)) \qquad m \in M.
\end{equation*}

\begin{lemma} \label{lem:cohomology} Let $h$ be a $\Delta$-linear support
  function with associated Cartier divisor $D_h$ and convex polytope
  $\square_h$ defined in \eqref{support}. The vector space
  ${\rm H}^0(X,O_X(D_h))$ of global sections of $O_X(D_h)$, i.e. rational
  functions $f$ on $X_{\square}$ such that $\mathrm{div}(f) + D_h \geq 0$ has
  dimension $\#(M \cap \square_{h})$ (the number of
  lattice points in $\square_{h}$) and has
  $\{{\bf e}(m) \vert\ m \in M \cap \square_h \} $ as a basis, see (\ref{char}).
\end{lemma}
\subsection{Polytopes, Cartier divisors and Intersection theory}
\label{mindist} For a fixed line bundle $\mathcal{L}$ on a variety $X$, given an
effective divisor $D$ such that $\mathcal{L}= O_X(D)$, the fundamental
question to answer is: How many points from a fixed set $\mathcal P$
of rational points are in the support of $D$.  This question is
treated in general in \cite{Hansen2001530} using intersection theory,
see \cite{MR1644323}. Here we will apply the same methods when $X$ is
a toric surface.

For a $\Delta$-linear support function $h$ and a 1-dimensional cone
$\rho \in \Delta (1)$, we will determine the intersection number
$(D_h;V(\rho))$ between the Cartier divisor $D_h$ and
$V(\rho)) =\mathbb P^1$. The cone $\rho$ is the common face of two
2-dimensional cones $\sigma', \sigma'' \in \Delta (2)$. Choose
primitive elements $n', n'' \in N$ such that
\begin{align*}
  n'+n''& \in \mathbb R \rho\\
  \sigma' + \mathbb R \rho &= \mathbb R_{\geq 0} n' + \mathbb R \rho\\
  \sigma'' + \mathbb R \rho &= \mathbb R_{\geq 0} n'' + \mathbb R \rho
\end{align*}
\begin{lemma}
  \label{lem:inter}For any $l_{\rho} \in M$, such that $h$ coincides with
  $l_{\rho}$ on $\rho$, let $\overline{h} = h-l_{\rho}$. Then the intersection number is
  \begin{equation*}intersection
    (D_h;V(\rho))=
    -\left(\overline{h}(n')+\overline{h}(n'')\right)\ .  
    \end{equation*}
\end{lemma}
In the 2-dimensional non-singular case, let $n(\rho)$ be a primitive
generator for the 1-dimensional cone $\rho$. There exists an integer
$a$ such that
\begin{equation*}
  n'+n''+a n(\rho)=0,
\end{equation*}
$V(\rho)$ is itself a Cartier divisor and the above gives the
self-intersection number
\begin{equation*}
  (V(\rho);V(\rho))=a.
\end{equation*}

\begin{lemma} \label{lem:self} Let $D_h$ be a Cartier divisor and let
  $\square_h$ be the polytope associated to $h$, see
  \eqref{support}. Then
  \begin{equation*}
    (D_h;D_h)= 2 {\rm vol}_{2}(\square_h),
  \end{equation*}
  where ${\rm vol}_{2}$ is the normalized Lebesgue-measure.
\end{lemma}

\section{Toric Codes}
\label{sec:toriccodes}

In this section we will present the construction of toric codes and the derivation of their parameters, see \cite{6ffeb030f4f511dd8f9a000ea68e967b},
\cite{39bd8e90f4f211dd8f9a000ea68e967b} and
\cite{53ee7c6020b511dcbee902004c4f4f50}.

\begin{definition}\label{def:def1} Let $M \simeq \Z^r$ be a free $\Z$-module of rank $r$ over the
integers $\Z$.

  For any subset $U \subseteq M$, let $\Fq[U]$ be the linear span in
  $\Fq[X_1^{\pm 1},\allowbreak \dots,\allowbreak X_r^{\pm 1}]$ of the monomials
  \begin{equation*}
    \{ X^u=X_1^{u_1}\cdot \dots \cdot X_r^{u_r} \vert\   u=(u_1,\dots,u_r)\in U \}\ .
  \end{equation*}
  This is an $\Fq$-vector space of dimension equal to the number of
  elements in $U$.

  Let $T(\Fq)=(\F)^r$ be the $\Fq$-rational points on the torus and
  let $S \subseteq T(\Fq)$ be any subset. The linear map that
  evaluates elements in $\Fq[U]$ at all the points in $S$ is denoted
  by $\pi_S$:
  \begin{align*}
    \pi_S:\Fq[U]&\rightarrow  \Fq^{\vert S \vert}\label{vectors}\\
    f&\mapsto(f(P))_{P\in S}\ .
  \end{align*}
  In this notation $\pi_{\{P\}}(f)=f(P)$.

  The toric code is the image $C_U=\pi_S(\Fq[U])$.
\end{definition}
\begin{remark}[($r=1$, Reed--Solomon code)]\label{RS}
  Consider the special case, where $M\simeq \Z$,
  $U=\square =\left[0,k-1\right] \subseteq M_{\mathbb R}= M
  \otimes_{\mathbb Z}{\mathbb R}$ and $S=T(\Fq)=\Fq^*$.

  The toric code $C_{\square}$ associated to $\square$ is the linear
  code of length $n=(q-1)$ presented in  Example \ref{def:RS} with
  $S=\{ x_1,\dots, x_n \}$.
\end{remark}
\begin{remark}[($r=2$)]\label{toric2}
  Consider the special case, where $M\simeq \Z^2$,
  $U=\square \subseteq M_{\mathbb R}= M \otimes_{\mathbb Z}{\mathbb
    R}$ is an integral convex polytope and
  $S=T(\Fq)=\Fq^* \times \Fq^*$.

  Let $\xi \in \mathbb F_q$ be a primitive element. For any $i$ such
  that $ 0 \leq i \leq q-1$ and any $j$ such that
  $0 \leq j \leq q-1 $, we let
  $P_{ij}=(\xi^i,\xi^j) \in S=\Fq^* \times \Fq^*$. Let ${m_1,m_2}$ be
  a $\mathbb Z$-basis for $M$.  For any
  $m =\lambda_1 m_1+\lambda_2 m_2 \in M \cap \square$, let
  ${\bf e}(m)(P_{ij}):=(\xi^i)^{\lambda_1}(\xi^j)^{\lambda_2}.$

  The toric code $C_{\square}$ associated to $\square$ is the linear
  code of length $n=(q-1)^2$ generated by the vectors
  \begin{equation*}
    \{ ({\bf e}(m)(P_{ij}))_{i=0,\dots,q-1;j=0,\dots,q-1} \vert \  m \in M \cap \square \}.
  \end{equation*}
\end{remark}

  The toric codes are  evaluation codes on the points
  of toric varieties.
  
  For each $t \in T(\Fq)=(\F)^r$ we evaluate the rational functions in
  ${\rm H}^0(X, O_X(D_h))$
  \begin{align*}
    {\rm H}^0(X, O_X(D_h))& \rightarrow  \Fbarq^{*} \\
    f & \mapsto  f(t).
  \end{align*}
  Let ${\rm H}^0(X, O_X(D_h))^{\textup{Frob}}$ denote the rational
  functions in ${\rm H}^0(X,O_X(D_h))$ that are invariant under
  the action of the Frobenius.  Evaluating in all points in $T(\Fq)$,
  we obtain the code $C_{\square}$:
  \begin{align*}
    {\rm H}^0(X, O_X(D_h))^{\textup{Frob}}& \rightarrow  C_{\square} \subset (\Fq^{*})^{\sharp T(\Fq)}
    \\
    f & \mapsto  (f(t))_{t \in T(\Fq)}\ ,
  \end{align*}
  as in Definition \ref{def:def1}.

\subsection{Hirzebruch surfaces and associated toric codes}

The Hirzebruch surfaces and associated toric codes will be used in (\ref{hirzLSSS}) to obtain Linear Secret Sharing Schemes. The toric surface constructed from then following polytope is a Hirzebruch surface, that is a $\mathbb P^1$ bundle over $\mathbb P^1$, see \cite{MR1234037}.

Let $d, e, r$ be positive integers and let $\square$ be the polytope
  in $M_{\mathbb R}$ with vertices $(0,0),(d,0),(d,e+rd),(0,e),$ see Figure
  \ref{fig:hirzpolytop}. Assume that $d<q-1$, that $e<q-1$ and that
  $e+rd<q-1$. We have that $n(\rho_1)=
  \begin{psmallmatrix}
    1\\0
  \end{psmallmatrix}
  $ , $n(\rho_2)=
  \begin{psmallmatrix}
    0\\1
  \end{psmallmatrix}
  $, $n(\rho_3)=
  \begin{psmallmatrix}
    -1\\0
  \end{psmallmatrix}
  $ and $n(\rho_4)=
  \begin{psmallmatrix}
    r\\-1
  \end{psmallmatrix}
  $. Let $\sigma_1$ be the cone generated by $n(\rho_1)$ and
  $n(\rho_2)$, $\sigma_2$ be the cone generated by $n(\rho_2)$ and
  $n(\rho_3)$ , $\sigma_3$ the cone generated by $n(\rho_3)$ and
  $n(\rho_4)$ and $\sigma_4$ the cone generated by $n(\rho_4)$ and
  $n(\rho_1)$. The support function is:
  \begin{equation*}
    h_{\square}
    \begin{psmallmatrix}
      n_1\\n_2
    \end{psmallmatrix}
    =
    \begin{cases}
      \begin{psmallmatrix}
        0\\0
      \end{psmallmatrix}
      .
      \begin{psmallmatrix}
        n_1\\n_2
      \end{psmallmatrix}
      & \text{if $
        \begin{psmallmatrix}
          n_1\\n_2
        \end{psmallmatrix}
        \in \sigma_1$},\\
      \begin{psmallmatrix}
        d\\0
      \end{psmallmatrix}
      .
      \begin{psmallmatrix}
        n_1\\n_2
      \end{psmallmatrix}
      & \text{if $
        \begin{psmallmatrix}
          n_1\\n_2
        \end{psmallmatrix}
        \in \sigma_2$} ,\\
      \begin{psmallmatrix}
        d\\e+rd
      \end{psmallmatrix}
      .
      \begin{psmallmatrix}
        n_1\\n_2
      \end{psmallmatrix}
      & \text{if $
        \begin{psmallmatrix}
          n_1\\n_2
        \end{psmallmatrix}
        \in \sigma_3$} ,\\
      \begin{psmallmatrix}
        0\\e
      \end{psmallmatrix}
      .
      \begin{psmallmatrix}
        n_1\\n_2
      \end{psmallmatrix}
      & \text{if $
        \begin{psmallmatrix}
          n_1\\n_2
        \end{psmallmatrix}
        \in \sigma_4$}.
    \end{cases}
  \end{equation*}

\begin{figure}
  \centering
  \includegraphics[height=4cm]{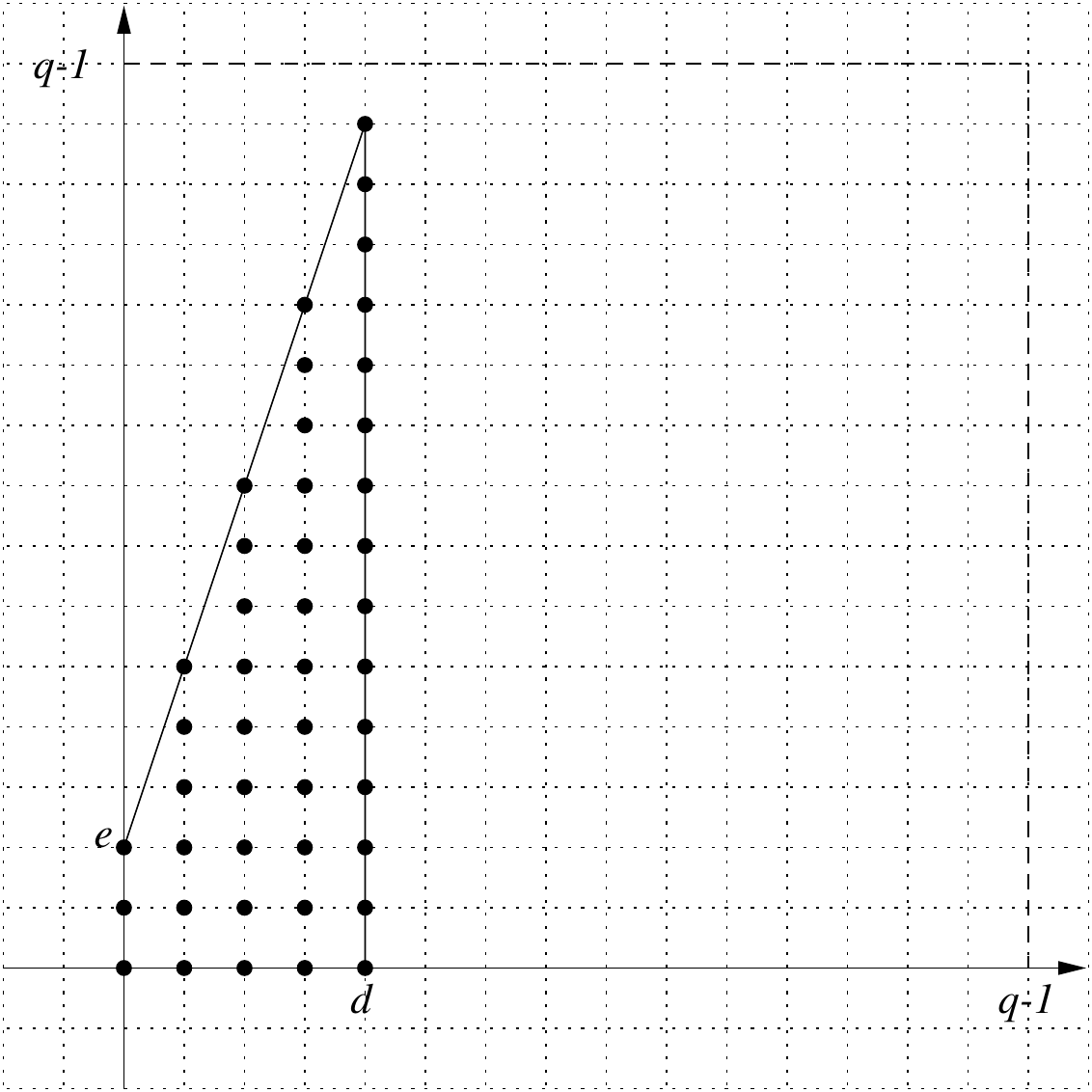}
\includegraphics[height=4cm]{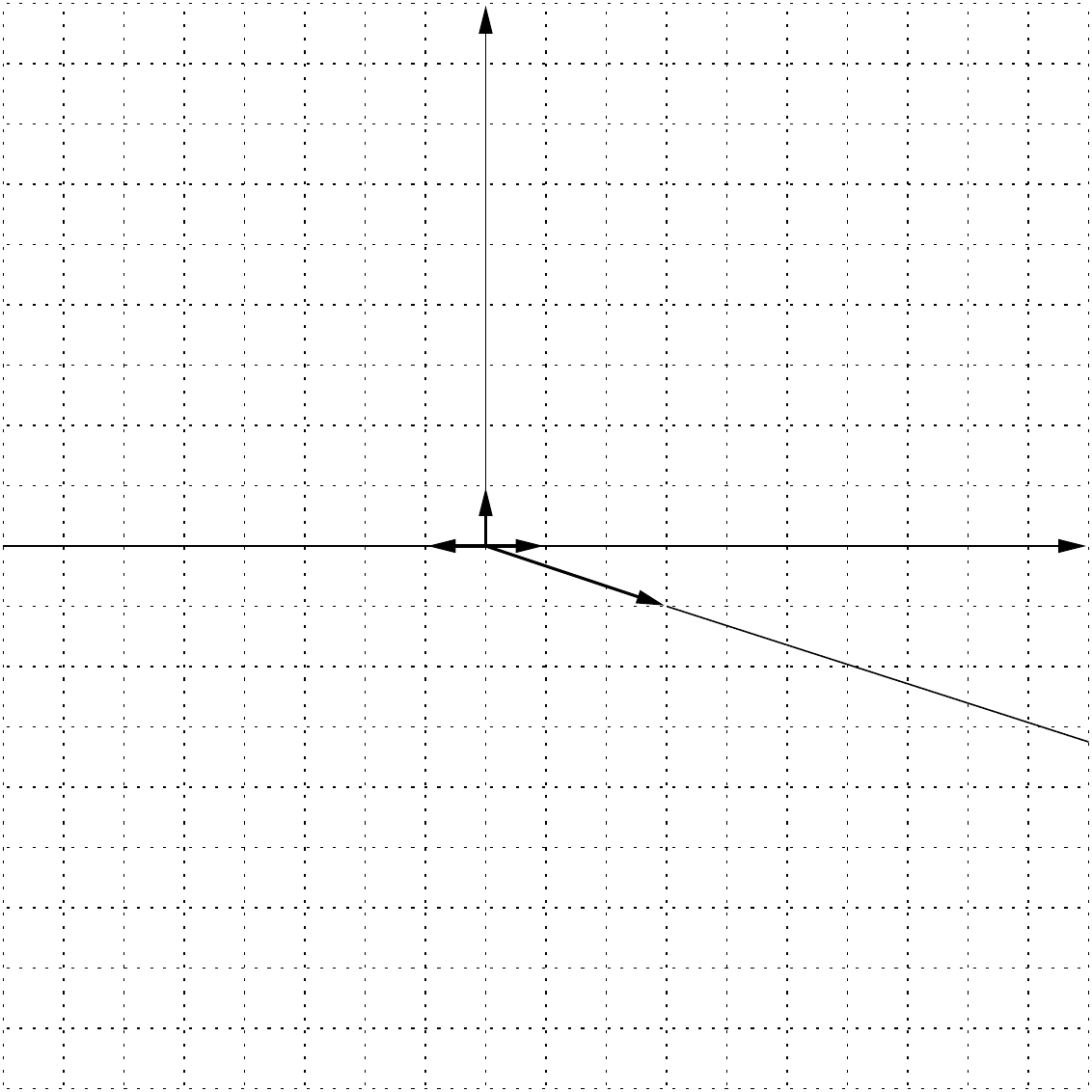}
  \caption{The polytope giving the Hirzebruch surface with
    vertices $(0,0),(d,0),(d,e+rd),(0,e).$ The associated (refined) normal fan with generators of the 1-dimensional cones.} \label{fig:hirzpolytop}
\end{figure}
  \begin{table}[ht]
  \begin{center}
    \caption{Intersection numbers for divisors on Hirzebruch surfaces, calculated using Lemma \ref{lem:inter} and
Lemma \ref{lem:self}. }
    \label{tabel:intersectiontabel}
  \begin{tabular}{ccccc}
    \hline
   &$V(\rho_1)$ & $V(\rho_2)$ & $V(\rho_3)$ & $V(\rho_4)$ \\
\hline
    $V(\rho_1)$& $-r$ & 1 & 0 & 1 \\
    $V(\rho_2)$ & 1 & 0 & 1 & 0 \\
    $V(\rho_3)$ & 0 & 1 & r & 1 \\
    $V(\rho_4)$ & 1 & 0 & 1 & 0 \\
\hline
  \end{tabular}
\end{center}
  \end{table}
  Also
  \begin{equation*}
  D_h:= -\sum_{\rho \in \Delta (1)} h(n(\rho))\,V(\rho) =  d
  \,V(\rho_3) +e \,V(\rho_4)
  \end{equation*}
and
\begin{equation*}
  \dim {\rm H}^0(X, O_X(D_h)) =
  (d+1)(e+1)+r \frac{d(d+1)}{2}.
\end{equation*}

  \begin{theorem} \label{thm:saetning4} Let $d, e, r$ be positive integers and
  let $\square$ be the polytope in $M_{\mathbb R}$ with vertices
  $(0,0),(d,0),(d,e+rd),(0,e),$ see Figure \ref{fig:hirzpolytop}. Assume
  that $d<q-1$, that $e<q-1$ and that $e+rd<q-1.$ The toric code
  $C_{\square}$ has 
  \begin{enumerate}
  \item length equal to $(q-1)^2$  
  \item dimension equal to $\# (M \cap \square)=(d+1)(e+1)+r\frac{d(d+1)}{2}$ (the number of
  lattice points in $\square$) and 
  \item minimal distance equal to
  \begin{equation*}
    {\rm min}\{(q-1-d)(q-1-e), (q-1)(q-1-e-rd)\}\ .
  \end{equation*}
\end{enumerate}
\end{theorem}
\begin{proof}

Let $m_{1}=(1,0) \in M_{\mathbb R}$. The $\Fq$-rational points of the torus
$T \simeq \Fbarq^{*} \times \Fbarq^{*}$ belong to the $q-1$ lines on
$X_{\square}$ given by $\prod_{\eta \in \F}({\bf e}(m_{1})-\eta)
=0$. Let $0 \neq f \in {\rm H}^0(X, O_X(D_h))$ and assume that $f$ is
zero along precisely $a$ of these lines. As ${\bf e}(m_{1})-\eta$ and
${\bf e}(m_{1})$ have the same divisors of poles, they have equivalent
divisors of zeroes, so
\begin{equation*}
  (\mathrm{div}({\bf e}(m_{1})-\eta))_{0} \sim (\mathrm{div}({\bf e}(m_{1})))_{0}.
\end{equation*}
Therefore
\begin{equation*}
  \mathrm{div}(f) + D_h -a (\mathrm{div}({\bf e}(m_{1})))_{0}\geq 0
\end{equation*}
or equivalently,
\begin{equation*}
  f \in {\rm H}^0(X, O_X(D_h-a (\mathrm{div}({\bf e}(m_{1})))_{0})).
\end{equation*}
This implies that $a \leq d$
according to Lemma \ref{lem:cohomology}. 

On any of the other $q-1-a$ lines,
the number of zeroes of $f$ is according to \cite{Hansen2001530} at
most the intersection number
\begin{equation}\label{intersection}
  (D_{h}-a (\mathrm{div}({\bf e}(m_{1})))_{0};(\mathrm{div}({\bf e}(m_{1})))_{0}) \ .
\end{equation}

The intersection number of (\ref{intersection}) is calculated
using Table  \ref{tabel:intersectiontabel}  as 
\begin{equation*}
(\mathrm{div}({\bf e}(m_{1})))_{0}=V(\rho_1)+r V(\rho_4)\ . 
\end{equation*}
We
get
\begin{equation*}
  \bigl(D_{h}-a (\mathrm{div} ({\bf e}(m_{1})))_{0};(\mathrm{div} ({\bf e}(m_{1})))_{0}\bigr)=e+(d-a)r.
\end{equation*}
As $0\leq a \leq d$, the total number of zeroes for $f$ is at most
\begin{align*}
  \MoveEqLeft a(q-1)+(q-1-a)(e+(d-a)r)
  \\
  &\leq \max\{d(q-1)+(q-1-d)e,(q-1)(e+dr)\}.
\end{align*}

This implies in both cases that the evaluation map
\begin{align*}
  {\rm H}^0(X, O_X(D_h))^{\textup{Frob}}& \rightarrow  C_{\square} \subset (\Fq)^{\vert T(\Fq) \vert}
  \\
  f & \mapsto  (f(t))_{t \in T(\Fq)}
\end{align*}
is injective and that the dimensions and the lower bounds for the
minimal distances of the toric codes are as claimed.

To see that the lower bounds for the minimal distances are in fact the
true minimal distances, we exhibit codewords of minimal weight.

Let
$b_1,\dots,b_{e+rd} \in \F$ be pairwise different elements.  Then
the function
\begin{equation*}
  x^d(y-b_1) \cdot \dots \cdot (y-b_{e+rd}) \in {\rm H} ^0(X, O_X(D_h))^{\textup{Frob}}
\end{equation*}
evaluates to zero in the $(q-1)(e+rd)$ points
\begin{equation*}
  (x,b_j), \quad x \in \F, \quad j=1,\dots,e+rd
\end{equation*}
and gives a codeword of weight
$(q-1)^2-(q-1)(e+rd)=(q-1)(q-1-(e+rd)).$ On the other hand, we let
$a_1,\dots,a_d \in \F$ be pairwise different elements and let
$b_1,\dots,b_e \in \F$ be pairwise different elements. Then the
function
\begin{equation*}
  (x-a_1)\cdot \dots \cdot(x-a_d)(y-b_1) \cdot \dots \cdot (y-b_e) \in
  {\rm H} ^0(X,O_X(D_h))^{\textup{Frob}} 
\end{equation*}
evaluates to zero in the $d(q-1)+(q-1)e-de$ points
\begin{equation*}
  (a_i,y),(x,b_j), \quad x,y \in \F, \quad i=1,\dots e, j=1,\dots,d
\end{equation*}
and gives a codeword of weight $(q-1-d)(q-1-e)$.
\end{proof}

\subsection{Some toric surfaces $X_{a,b}$ and their associated toric codes $C_{a,b}$}\label{X_{a,b}}

We present some toric surfaces and associated toric codes which we in (\ref{strong})
use to construct Linear Secret Sharing Schemes with strong multiplication and in 
Section \ref{QC} to construct quantum codes.

Let $a, b$ be positive integers $0 \leq b \leq a \leq q-2$, and let
$\square$ be the polytope in $M_{\mathbb R}$ with vertices $(0,0)$,
$(a,0)$, $(b,q-2)$, $(0,q-2)$ rendered in Figure \ref{fig:polytope} and with
normal fan as in Figure \ref{fig:cones}.

The primitive generators of the 1-dimensional cones are
\begin{equation*}
  n(\rho_1)=\begin{pmatrix}1\\0\end{pmatrix},\quad 
  n(\rho_2)=\begin{pmatrix}0\\1\end{pmatrix},\quad 
  n(\rho_3)=\begin{pmatrix}\frac{-(q-2)}{\gcd(a-b,q-2)}\\\frac{-(a-b)}{\gcd(a-b,q-2)}\end{pmatrix},\quad
  n(\rho_4)=\begin{pmatrix}0\\-1\end{pmatrix}.
\end{equation*} 

For $i=1,\dots, 4$, the 2-dimensional cones $\sigma_i$ are shown in Figure 
\ref{fig:cones}. The faces of $\sigma_1$ are $\{0, \rho_1, \rho_2\}$,
the faces of $\sigma_2$ are $\{0, \rho_2, \rho_3\}$, the faces of
$\sigma_3$ are $\{0, \rho_3, \rho_4\}$ and the faces of $\sigma_4$ are
$\{0, \rho_4, \rho_1\}$.

The support function of $\square$ is
\begin{equation}\label{supportfct}
  h \begin{psmallmatrix}n_1\\n_2\end{psmallmatrix}=
  \begin{cases}
    \begin{psmallmatrix}0\\0\end{psmallmatrix} . \begin{psmallmatrix}n_1\\n_2 \end{psmallmatrix} & \text{if $\begin{psmallmatrix}n_1\\n_2\end{psmallmatrix} \in \sigma_1$},\\
    \begin{psmallmatrix}a\\0\end{psmallmatrix} . \begin{psmallmatrix}n_1\\n_2 \end{psmallmatrix} & \text{if $\begin{psmallmatrix}n_1\\n_2\end{psmallmatrix} \in \sigma_2$} ,\\
    \begin{psmallmatrix}b\\q-2\end{psmallmatrix} . \begin{psmallmatrix}n_1\\n_2 \end{psmallmatrix} & \text{if $\begin{psmallmatrix}n_1\\n_2\end{psmallmatrix} \in \sigma_3$} ,\\
    \begin{psmallmatrix}0\\q-2\end{psmallmatrix}
    . \begin{psmallmatrix}n_1\\n_2 \end{psmallmatrix} & \text{if
      $\begin{psmallmatrix}n_1\\n_2\end{psmallmatrix} \in \sigma_4$}.
  \end{cases}
\end{equation}

The related toric surface is in general singular as
$\{n(\rho_2), n(\rho_3)\}$ and $\{n(\rho_3), n(\rho_4)\}$ are not
bases for the lattice $M$. We can desingularize by subdividing the
cones $\sigma_2$ and $\sigma_3$, however, our calculations will only
involve the cones $\sigma_1$ and $\sigma_2$, so we refrain from that.

\begin{theorem} \label{thm:saetning3}
  Assume $a, b$ are integers with $0 \leq b \leq a \leq q-2$.

  Let $\square$ be the polytope in $M_{\mathbb R}$ with vertices
  $(0,0)$, $(a,0)$, $(b,q-2)$, $(0,q-2)$ rendered in Figure
  \ref{fig:polytope}, and let $U= M \cap \square$ be the lattice
  points in $\square$.
  
  \begin{enumerate}
  
  \item\label{item:thm5:1} The maximum number of zeros of
    $\pi_{T(\Fq)}(f)$ for $f\in \Fq[U]$ is less than or equal to
    \begin{equation*}
      (q-1)^2-(q-1-a).
  \end{equation*}
\item\label{iitem:thm5:2} The toric code $C_{\square}$ has 
\begin{enumerate}
\item length equal to $(q-1)^2$ ,
\item dimension equal to $\# (M \cap \square)=\frac{(q-1)(a+b+1)+\gcd(a-b,q-2)+1}{2}$ (the
  number of lattice points in $\square$) and
  \item minimal distance greater than or equal to $q-1-a$\ .
  \end{enumerate}
\end{enumerate}

\end{theorem}

\begin{proof}
  Let $m_{1}=(1,0)$.  The $\Fq$-rational points of
  $T \simeq \Fbarq^{*} \times \Fbarq^{*}$ belong to the $q-1$ lines on
  $X_{\square}$ given by
  \begin{equation*}
    \prod_{\eta \in \F}({\bf e}(m_{1})-\eta) =0 .
  \end{equation*}
  Let $0 \neq f \in {\rm H}^0(X,O_X(D_h))$. Assume that $f$ is zero along
  precisely $c$ of these lines.

  As ${\bf e}(m_{1})-\eta$ and ${\bf e}(m_{1})$ have the same divisors of
  poles, they have equivalent divisors of zeroes, so
  \begin{equation*}
    ({\bf e}(m_{1})-\eta)_{0} \sim ({\bf e}(m_{1}))_{0}.
  \end{equation*}
  Therefore
  \begin{equation*}
    \mathrm{div}(f) + D_h -c ({\bf e}(m_{1}))_{0}\geq 0
  \end{equation*}
  or equivalently
  \begin{equation*}
    f \in {\rm H} ^0(X,O_X(D_h-c ({\bf e}(m_{1}))_{0})).
  \end{equation*}
  This implies that $c \leq a$ according to Lemma \ref{lem:cohomology}.

  On any of the other $q-1-c$ lines the number of zeroes of $f$ is at
  most the intersection number
  \begin{equation*}
    (D_{h}-c ({\bf e}(m_{1}))_{0};({\bf e}(m_{1}))_{0}).
  \end{equation*}
 
  This number can be calculated using Lemma \ref{lem:inter} as $({\bf e}(m_{1}))_{0}=V(\rho_1)$ and
  \begin{equation*}
  n(\rho_2)+n(\rho_4) + 0 \cdot n(\rho_1)=0\ ,
\end{equation*}
so 
\begin{equation}\label{self1}({\bf e}(m_{1}))_{0};({\bf e}(m_{1}))_{0})=0\ .
\end{equation}

  We get from \eqref{supportfct} and \eqref{self1} that
  \begin{align*}
    \MoveEqLeft (D_{h}-c ({\bf e}(m_{1}))_{0};({\bf e}(m_{1}))_{0})
    \\
    &=(D_{h};({\bf e}(m_{1}))_{0}) - c
    ({\bf e}(m_{1}))_{0};({\bf e}(m_{1}))_{0})
    \\
    &=- h_{\square}\begin{psmallmatrix}0
      \\
      1\end{psmallmatrix}-h_{\square}\begin{psmallmatrix}0
      \\
      -1\end{psmallmatrix} = q-2\ .
  \end{align*}

  As $0\leq c \leq a$, we conclude the total number of zeroes for $f$
  is at most
  \begin{equation*}
    c(q-1)+(q-1-c)(q-2) \leq a(q-1)+(q-1-a)(q-2)= (q-1)^2-(q-1-a)
  \end{equation*}
  proving \ref{item:thm5:1}.
  
 The claims in \ref{item:thm5:2} follows from \ref{item:thm5:1} counting the number of lattice points.
\end{proof}

\begin{figure}
  \centering
  \begin{tikzpicture}[scale=0.6]
      \draw[very thick] (0,7.5)--(15,7.5);
      \draw[very thick] (15,7.5)--(15,15)--(7.5,15)--(7.5,7.5);
      \draw[very thick] (7.5,0)--(7.5,15);
     
      \draw (8.5,7.5) node[anchor=north]{$b$};
      \draw (10.5,7.5) node[anchor=north]{$a$};
      \draw(8.5,9.5) node[anchor=north]{$\square$};
      \draw(14,14.5) node[anchor=north]{$H$};
     
      \draw[thick,fill=black!20] (7.5,7.5)--(4.5,7.5)--(6.5,0)--(7.5,0)--(7.5,7.5);
      \draw (4.5,7.5) node[anchor=south]{$-a$};
      \draw (6.5,7.5) node[anchor=south]{$-b$};
      \draw(6.5,5.5) node[anchor=north]{$-\square$};
      \draw[thick,fill=black!10] (4.5,7.5)--(0,7.5)--(0,0)--(6.5,0)--(4.5,7.5);
      \draw(3.5,3.5) node[anchor=north]{$-H \setminus -\square$};
      \draw[thick,fill=black!10] (12,15)--(7.5,15)--(7.5,7.5)--(14,7.5)--(12,15);
      \draw[thick,fill=black!20] (7.5,7.5)--(13.5,7.5)--(11.5,15)--(7.5,15)--(7.5,7.5);
      \draw[thick,fill=black!30] (7.5,7.5)--(10.5,7.5)--(8.5,15)--(7.5,15)--(7.5,7.5);
      \draw(8.5,9.5) node[anchor=north]{$\square$};
      \draw(12.9,13.6) node[anchor=north]{$(q-2,q-2)+\big(-H \setminus -\square \big)$};
      \draw (13.5,5.5) node[anchor=north,rotate=90]{$(q-2)-b$};
      \draw (11.5,5.5) node[anchor=north,rotate=90]{$(q-2)-a$};
      \draw (13.5,7.5) node[anchor=north]{$2a$};
      \draw[step=.5cm,gray,very thin,dashed]
      (0,0) grid (15,15);
      \draw(1,1) node[anchor=north]{$-H$};
      \draw (7.5,15) node[anchor=east]{$q-2$};
      \draw (7.5,0) node[anchor=west]{$-(q-2)$};
       \draw (14.6,6.5) node[anchor=north,rotate=90]{$q-2$};
      \draw (0.4,8.5) node[anchor=south,rotate=90]{$-(q-2)$};
    \end{tikzpicture}
    \caption{The convex polytope $H$ with vertices
      $(0,0), (q-2,0), (q-2,q-2), (0,q-2)$ and the convex polytope
      $\square$ with vertices $(0,0), (a,0), (b,q-2), (0,q-2))$ are
      shown. Also their opposite convex polytopes $-H$ and $-\square$,
      the complement $ -H \setminus -\square$ and its translate
      $(q-2,q-2)+(-H \setminus -\square )$ are depicted.  Finally the
      convex hull of the reduction modulo $q-1$ of the Minkowski sum
      $U+U$ of the lattice points $U= \square \cap M$ in $\square$, is
      rendered. It has vertices $(0,0)$, $(2a,0)$, $(a+b,q-2)$ and
      $(0,q-2)$.}\label{fig:polytope}
\end{figure}
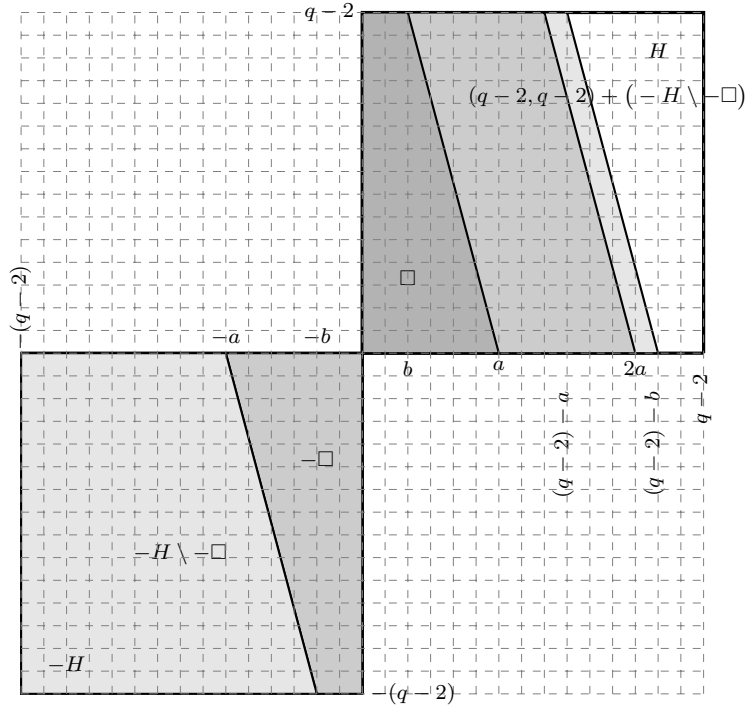

\begin{figure}
  \centering
  \begin{tikzpicture}[scale=0.4]
      \draw[thick] (0,7.5)--(15,7.5);
      \draw[thick] (7.5,0)--(7.5,15);
      \draw[very thick] (7.5,7.5)--(7.5,15);
      \draw[very thick] (7.5,7.5)--(7.5,0);
      \draw[very thick] (7.5,7.5)--(15,7.5);
      \draw[very thick] (7.5,7.5)--(0,5.5);
      \draw[->, very thick] (7.5,7.5)--(7.5,8);
      \draw[->, very thick] (7.5,7.5)--(7.5,7);
      \draw[->, very thick] (7.5,7.5)--(8,7.5);
      \draw[->, very thick] (7.5,7.5)--(0,5.5);
      \draw (7.3,9.5) node[anchor=east] {$\rho_2$};
      \draw (2,7) node[anchor=north] {$\rho_3$};
      \draw (9.5,8.3) node[anchor=north] {$\rho_1$};
      \draw (8,5.8) node[anchor=north] {$\rho_4$};

      \draw (11,10.5) node[anchor=east] {$\sigma_1$};
      \draw (5,10.5) node[anchor=east] {$\sigma_2$};
      \draw (11,4.5) node[anchor=east] {$\sigma_4$};
      \draw (5,4.5) node[anchor=east] {$\sigma_3$};
      \draw (3.6,3) node[anchor=north] {$n(\rho_3)=\big(\frac{-(q-2)}{\gcd(a-b,q-2)},\frac{-(a-b)}{\gcd(a-b,q-2)}\big)$};
      \draw[step=.5cm,gray,very thin,dashed](0,0) grid (15,15);
    \end{tikzpicture}
    \caption{The normal fan and its 1-dimensional cones $\rho_i$, with
      primitive generators $n(\rho_i)$, and 2-dimensional cones
      $\sigma_i$ for $i=1,\dots, 4$ of the polytope $\square$ in Figure
      \ref{fig:polytope}.}\label{fig:cones}
\end{figure}
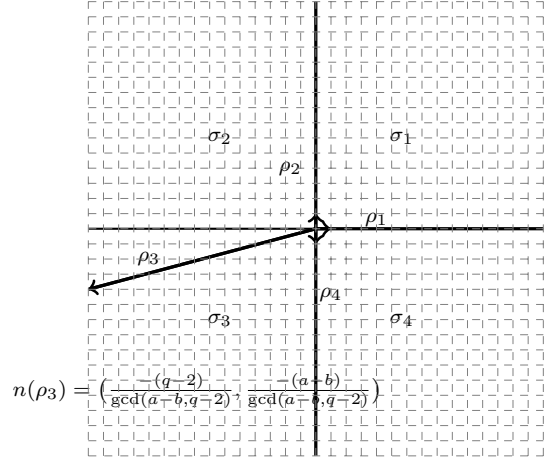

\subsection{Dual toric code}\label{sec:dualtoric}
Let $M \simeq \Z^r$ be a free $\Z$-module of rank $r$ over the
integers $\Z$. Let $U \subseteq M$ be a subset, let $v \in M$ and consider
translation $v + U:=\{ v+u \vert \  u \in U \} \subseteq M$.
\begin{lemma}\label{lem:trans}
  Translation induces an isomorphism of vector spaces
  \begin{align*}
    \Fq[U]&\rightarrow \Fq[v+U]\\
    f&\mapsto f^v:=X^v \cdot f .
  \end{align*}
  We have that
  \begin{enumerate}
  \item\label{item:lem:trans:1} The evaluations of $\pi_{T(\Fq)}(f)$
    and $\pi_{T(\Fq)}(f^v)$ have the same number of zeroes on
    $T(\Fq)$.
  \item\label{item:lem:trans:2} The minimal number of zeros on
    $T(\Fq)$ of evaluations of elements in $\Fq[U]$ and $\Fq[v+U]$ are
    the same.
  \item\label{item:lem:trans:3} For $v=(v_1,\dots,v_r)$ with $v_i$
    divisible by $q-1$, the evaluations $\pi_S(f)$ and $\pi_S(f^v)$
    are the same for any subset $S$ of $T(\Fq)$.
  \end{enumerate}
\end{lemma}
The lemma and generalizations have been used in several articles
classifying toric codes, e.g., \cite{MR2272243}.

As an immediate consequence of \ref{item:lem:trans:3} above, see also \cite[Theorem~3.3]{MR2360532}, we have:
\begin{corollary}
  \label{cor:reduction}
  Let $U \subseteq M$ be a subset and let
  \begin{equation*}
    \bar{U}:=\{ (\bar{u}_1,\dots,\bar{u}_r) \vert \    \bar{u}_i\in
      \{0,\dots,q-2 \}  \text{ and } \bar{u}_i \equiv u_i \bmod q-1 \}
  \end{equation*}
  be its reduction modulo $q-1$.  Then
  $\pi_S(\Fq[U]) =\pi_S(\Fq[\bar{U}])$ for any subset
  $S \subseteq T(\Fq)$.
\end{corollary}

Proposition \ref{prop:orto} exhibits the dual code of the toric code
$C=\pi_S(\Fq[U])$ defined in Definition \ref{def:def1}.

Let $U \subseteq M$ be a subset, define its opposite as
$-U:=\{ -u \vert \   u \in U \} \subseteq M$.  The opposite maps the
monomial $X^u$ to $X^{-u}$ and induces by linearity an isomorphism of
vector spaces
\begin{align*}
  \Fq[U]&\rightarrow \Fq[-U]
  \\
  X^u&\mapsto \hat{X^u}:=X^{-u}
  \\
  f&\mapsto\hat{f}.
\end{align*}
On $\Fq^{\vert T(\Fq)\vert }$, we have the inner product
\begin{equation*}
  (a_0,\dots,a_n)\star (b_0,\dots,b_n)=\sum_{l=0}^n a_l b_l \in \Fq,
\end{equation*}
with $n=\vert {T(\Fq)}\vert -1$.
\begin{lemma}
  Let $f,g \in \Fq [M] $ and assume $f \neq \hat{g}$, then
  \begin{equation*}
    \pi_{T(\Fq)}(f)\star \pi_{T(\Fq)}(g)= 0.
  \end{equation*}
\end{lemma}

With this inner product we obtain the following proposition,
e.g. \cite[Proposition~3.5]{MR2377234} and
\cite[Theorem~6]{MR2499927}.
\begin{proposition}\label{prop:orto} Let
\begin{equation*}
  H =\{0,1,\dots,q-2\}\times \dots \times \{0,1,\dots,q-2\} \subset M\ .
\end{equation*}
  Let $U \subseteq H$ be a subset. Then we have
  \begin{enumerate}
  \item For $f \in \Fq[U]$ and $g \notin \Fq[-H\setminus -U]$, we have
    that $\pi_{T(\Fq)}(f)\star \pi_{T(\Fq)}(g)=0$.

  \item The orthogonal complement to $\pi_{T(\Fq)}(\Fq[U])$ in
    $\Fq^{\vert T(\Fq)\vert}$ is
    \begin{equation*}
      \pi_{T(\Fq)}(\Fq[-H \setminus -U]) ,
    \end{equation*}
    i.e., the dual code of $C=\pi_{T(\Fq)}(\Fq[U])$ is
    $\pi_{T(\Fq)}(\Fq[-H \setminus -U])$.
  \end{enumerate}
\end{proposition}
Examples are  shown in Figures \ref{fig:Hirz} and  \ref{fig:polytope}.
\begin{figure}
  \centering
  \begin{tikzpicture}[scale=0.5]
    \draw[very thick] (0,7.5)--(15,7.5);
    \draw[very thick] (7.5,0)--(7.5,15);
    \draw[thick,fill=black!30] (7.5,7.5)--(9.5,7.5)--(9.5,13.5)--(7.5,10.5)--(7.5,7.5);
    \draw[thick,fill=black!20] (7.5,7.5)--(5.5,7.5)--(5.5,1.5)--(7.5,4.5)--(7.5,7.5);
    \draw[thick,fill=black!10] (5.5,7.5)--(5.5,1.5)--(7.5,4.5)--(7.5,0)--(0,0)--(0,7.5)--(5.5,7.5);
    \draw[thick] (7.5,7.5)--(15,7.5)--(15,15)--(7.5,15)--(7.5,7.5);
    \draw (9.5,7.5) node[anchor=north]{$d$};
    \draw (7.5,10.5) node[anchor=east]{$e$};
    \draw (7.5,13.5) node[anchor=east]{$e+rd$};
    \draw(8.5,9.5) node[anchor=north]{$\square$};
    \draw(13.5,10.5) node[anchor=north]{$H$};
    \draw(6.5,6.5) node[anchor=north]{$-\square$};
    \draw(2.5,3.5) node[anchor=north]{$-H\setminus-\square$};
    \draw[step=.5cm,gray,thin,dashed](0,0) grid (15,15);
    \draw (7.5,15) node[anchor=east]{$q-2$};
    \draw (7.5,0) node[anchor=west]{$-(q-2)$};
    \draw (14.5,6.5) node[anchor=north,rotate=90]{$(q-2)$};
    \draw (0.5,8.5) node[anchor=south,rotate=90]{$-(q-2)$};
  \end{tikzpicture}
  \caption{Hirzebruch surfaces. The convex polytope $H$ with vertices
    $(0,0)$, $(q-2,0)$, $(q-2,q-2)$, $(0,q-2)$. The convex polytope
    $\square$ with vertices $(0,0)$, $(d,0)$, $(d,e+rd)$, $(0,e)$ and
    their opposite convex polytopes $-H$ and $-\square$. Also the
    (non-convex) polytope $-H\setminus -\square$ is
    depicted.}\label{fig:Hirz}
\end{figure}
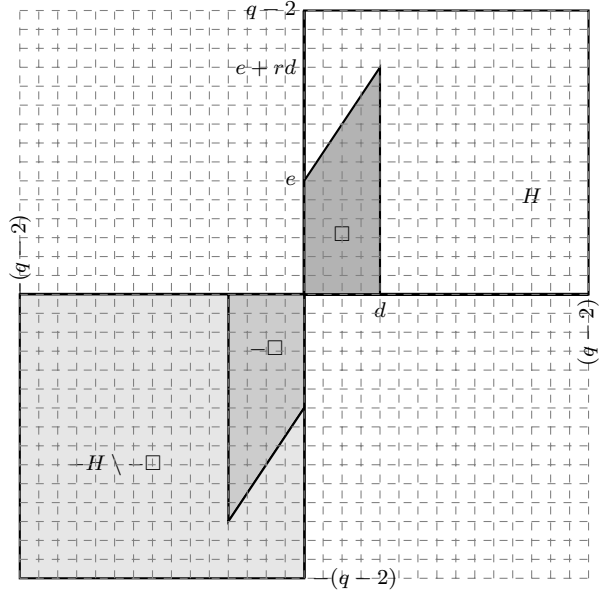

\section{Secret Sharing Schemes from toric varieties and
  codes}\label{LSSStoric}

Secret sharing schemes were introduced in \cite{Blakley1979} and
\cite{journals/cacm/Shamir79} and provide a method to split a
\emph{secret} into several pieces of information (\emph{shares}) such
that any large enough subset of the shares determines the secret,
while any small subset of shares provides no information on the
secret.

J. L. Massey presented in \cite{MR2017562} a method for constructing linear secret sharing schemes from
error-correcting codes. We  apply his method to toric  codes. 

The schemes obtained are ideal and the number of players is
$(q-1)r-1$ for any positive integer $r$. Examples of schemes which are
quasi-threshold and have strong multiplication with respect to certain
adversary structures are also presented. 

Secret sharing schemes have found applications in cryptography, when
the schemes have certain algebraic properties. 

\emph{Linear secret
  sharing schemes} (LSSS) are schemes where the secrets $s$ and their
associated shares $(a_1,\dots, a_n)$ are elements in a vector space
over $\Fq$. The schemes are called
\emph{ideal} if the secret $s$ and the shares $a_i$ are elements in
that ground field $\Fq$. Specifically, if $s, \tilde{s} \in \Fq$ are
two secrets with share vectors
$(a_1,\dots a_n)$, $(\tilde{a}_1,\dots \tilde{a}_n) \in \Fq^n$, then the
share vector of the secret $s+\lambda \tilde{s} \in \Fq$ is
$(a_1+\lambda \tilde{a}_1,\dots, a_n+ \lambda \tilde{a}_n) \in \Fq^n$
for any $\lambda \in \Fq$.

\emph{The reconstruction threshold} of the linear secret sharing
scheme is the smallest integer $r$ such that any set of at least $r$
of the shares $a_1,\dots,a_n$ determines the secret $s$. The
\emph{privacy threshold} is the largest integer $t$ such that no set
of $t$ (or fewer) elements of the shares $a_1,\dots,a_n$ determines
the secret $s$. The scheme is said to have $t$-\emph{privacy}.

An ideal linear secret sharing scheme is said to have
\emph{multiplication} if the product of the shares determines the
product of the secrets. It has $t$-\emph{strong multiplication} if it
has $t$-\emph{privacy} and has multiplication for any subset of $n-t$
shares obtained by removing any $t$ shares.

The properties of multiplication was introduced in \cite{CDM}. Such
schemes with multiplication can be utilized in the domain of
multiparty computation (MPC), see \cite{DBLP:conf/stoc/ChaumCD88},
\cite{DBLP:conf/stoc/Ben-OrGW88}, \cite{DBLP:books/cu/CDN2015} and
\cite{17425}.

\subsection{Basic definitions and concepts -- Linear Secret Sharing
  Schemes (LSSS)}

This section presents basic definitions and concepts pertaining to
linear secret sharing schemes as introduced in
\cite{MR2017562}, \cite{CDM}, \cite{CC} and \cite{MR2449216}.

Let $\Fq$ be a finite field with $q$ elements.

An \emph{ideal linear secret sharing scheme} $\mathcal M$ over a
finite field $\Fq$ on a set $\mathcal{P}$ of $n$ players is given by a
positive integer $e$, a sequence $V_1,\dots V_n$ of 1-dimensional
linear subspaces $V_i \subset \Fq^e$ and a non-zero vector
$u \in \Fq^e$.

An \emph{adversary structure} $\mathcal{A}$, for a secret sharing
scheme $\mathcal{M}$ on the set of players~$\mathcal{P}$, is a
collection of subsets of $\mathcal{P}$, with the property that subsets
of sets in $\mathcal{A}$ are also sets in $\mathcal{A}$. In
particular, the \emph{adversary structure} $\mathcal{A}_{t,n}$
consists of all the subsets of size at most $t$ of the set
$\mathcal{P}$ of $n$ players, and the \emph{access structure}
$\Gamma_{r,n}$ consists of all the subsets of size at least $r$ of the
set $\mathcal{P}$ of $n$ players.

For any subset $A$ of players, let $V_A= \sum_{i\in A} V_i$ be the
$\Fq$-subspace spanned by all the $V_i$ for $i \in A$.

The \emph{access structure} $\Gamma(\mathcal{M})$ of $\mathcal{M}$
consists of all the subsets $B$ of players with $u \in V_B$, and
$\mathcal{A}(\mathcal{M})$ consists of all the other subsets $A$ of
players, that is $A \notin \Gamma(\mathcal{M})$.

A linear secret sharing scheme $\mathcal M$ is said to \emph{reject} a
given adversary structure $\mathcal{A}$, if
$\mathcal{A} \subseteq \mathcal{A}(\mathcal{M})$. Therefore
$A \in \mathcal{A}(\mathcal{M})$ if and only if there is a linear map
from $\Fq^e$ to $\Fq$ vanishing on $V_A$, while non-zero on $u$.

The scheme $\mathcal M$ works as follows. For $i = 1,\dots n$, let
$v_i \in V_i$ be bases for the 1-dimensional vector spaces. Let
$s \in \Fq$ be a \emph{secret}. Choose at random a linear morphism
$\phi : \Fq^e \rightarrow \Fq$, subject to the condition $\phi(u)=s$,
and let $a_i = \phi(v_i)$ for $i = 1,\dots, n$ be the \emph{shares}
\begin{align*}
  \phi : \Fq^e& \rightarrow \Fq 
  \\
  u& \mapsto s
  \\
  v_i& \mapsto  a_i,  i = 1,\dots, n.
\end{align*}

Then
\begin{itemize}
\item the shares $\{a_i=\phi(v_i)\}_{i \in A}$ determine the secret
  $s=\phi(u)$ uniquely if and only if $A \in \Gamma(\mathcal{M})$,
\item the shares $\{a_i=\phi(v_i)\}_{i \in A}$ reveal no information
  on the secret $s=\phi(u)$, i.e., when
  $A \in \mathcal{A}(\mathcal{M})$.
\end{itemize}

\begin{definition} \label{def:defthresholds} Let $\mathcal M$ be a linear secret
  sharing scheme.

  The \emph{reconstruction threshold} of $\mathcal M$ is the smallest
  integer $r$ so that any set of at least $r$ of the shares
  $a_1,\dots,a_n$ determines the secret $s$, i.e.,
  $\Gamma_{r,n} \subseteq \Gamma(\mathcal{M})$.

  The \emph{privacy threshold} is the largest integer $t$ so that no
  set of $t$ (or less) elements of the shares $a_1,\dots,a_n$
  determine the secret $s$, i.e.,
  $\mathcal{A}_{t,n}\subseteq \mathcal{A}(\mathcal{M})$.  The scheme
  $\mathcal M$ is said to have $t$-\emph{privacy}.
\end{definition}

\begin{definition}\label{def:defstrong}
  An ideal linear secret sharing scheme $\mathcal{M}$ has the
  \emph{strong multiplication property} with respect to an adversary
  structure $\mathcal{A}$ if the following holds.
  \begin{enumerate}
  \item $\mathcal{M}$ rejects the adversary structure $\mathcal{A}$.
  \item Given two secrets $s$ and $\tilde{s}$. For each
    $A \in \mathcal{A}$, the products $a_i \cdot \tilde{a}_i$ of all
    the shares of the players $i \notin A$ determine the product
    $s \cdot \tilde{s}$ of the two secrets.
  \end{enumerate}
\end{definition}

\subsection{Secret sharing schemes from toric codes using the construction of J. L.  Massey}

Linear secret sharing schemes obtained from linear codes were
introduced by James L. Massey in \cite{MR2017562} and were generalized
in \cite[Section~4.1]{MR2449216}. Specifically, a linear secret sharing scheme with $n$ players is
obtained from a linear $C$ code of length $n+1$ and dimension $k$ with
privacy threshold $t = d'-2$ and reconstruction threshold $r= n-d+2$,
where $d$ is the minimum distance of the code and $d'$ the minimum
distance of the dual code.

This method of toric varieties also applies to construct algebraic
geometric ideal secret sharing schemes (LSSS) defined over a finite
ground field $\Fq$ with $q$ elements. In a certain sense our
construction resembles that of \cite{CC}, where LSSS schemes were
constructed from algebraic curves. However, the methods of obtaining
the parameters are completely different.

The linear secret sharing schemes we obtain are \emph{ideal} and the
number of players can be of the magnitude $q^r$ for any positive
integer $r$. 

The thresholds and conditions for strong multiplication are derived
from estimates on the maximal number of zeroes of rational functions
obtained via the cohomology and intersection theory on the underlying
toric variety. In particular, we focus on toric surfaces.

\subsection{The construction of Linear Secret Sharing Schemes (LSSS) from toric varieties}

With notation as in Definition \ref{def:def1}.
\begin{definition}\label{def:LSSS}
  Let $S \subseteq T(\Fq)$ be any subset so that $P_0 \in S$.  The
  linear secret sharing schemes (LSSS) $\mathcal{M}(U)$ with
  \emph{support} $S$ and $n=\vert {S}\vert -1$ players is obtained as
  follows:
  \begin{itemize}
  \item Let $s_0 \in \Fq$ be a \emph{secret} value. Select
    $f \in \Fq[U]$ at random, such that
    $\pi_{\{P_0\}}(f)=f(P_0)= s_0$.
  \item Define the $n$ shares as the evaluations
    \begin{equation*}
      \pi_{S\setminus{\{P_0\}}}(f)= (f(P))_{P\in S\setminus{\{P_0\}}} \in \Fq^{\vert S \vert-1} = \Fq^n.
    \end{equation*}
  \end{itemize}
\end{definition}

\begin{theorem} \label{thm:thresholds}
Let $\mathcal{M}(U)$ be the linear secret sharing schemes of
  Definition \ref{def:LSSS} with $(q-1)^r-1$ players.
  
  Let $r(U)$ and $t(U)$ be the reconstruction and privacy thresholds
  of $\mathcal{M}(U)$ as defined in Definition \ref{def:defthresholds}.

  Then
  \begin{align*}
    r(U) &\geq  (\text{the maximum number of zeros of $\pi_{T(\Fq)}(f)$})+2
    \\
     t(U) & \leq  (q-1)^r-(\text{the maximum number of zeros of $\pi_{T(\Fq)}(g)$})-2
    \end{align*}
  for some $f\in \Fq[U]$ and for some $g\in \Fq[-H\setminus -U]$,
  where
  \begin{align*}
    \pi_{T(\Fq)}:\Fq[U]&\rightarrow \Fq^{\vert T(\Fq) \vert}
    \\
    f&\mapsto \pi_{T(\Fq)}(f)=(f(P))_{P\in {T(\Fq)}}
    \\
    \pi_{T(\Fq)}:\Fq[-H\setminus -U]&\rightarrow \Fq^{\vert T(\Fq) \vert}
    \\
    g&\mapsto\pi_{T(\Fq)}(g)=(g(P))_{P\in {T(\Fq)}} .
  \end{align*}
\end{theorem}

\begin{proof}The minimal distance of an evaluation code and the maximum
  number of zeros of a function add to the length of the code.

  The bound for $r(U)$ is based on the minimum distance $d$ of the
  code
  $C= \pi_{T(\smash{\Fq})}(\Fq[U])\subseteq \Fq^{\vert T(\Fq) \vert}$,
  the bound for $t(U)$ is base on the minimum distance $d'$ of the
  dual code
  $C'=\pi_{T(\Fq)}(\Fq[-H\setminus -U] \subseteq
  \Fq^{\vert T(\Fq) \vert}$, using Proposition \ref{prop:orto} to represent
  the dual code as an evaluation code.

  The codes have length $\vert T(\Fq) \vert$, hence,
  \begin{align*}
    r(U) &\geq  \vert T(\Fq) \vert-d+2
    \\
    &=(\text{the maximum number of
      zeros of $\pi_{T(\Fq)}(f)$}) +2
    \\
    t(U) &\leq  d'-2
    \\
    &= \vert T(\Fq) \vert-(\text{the maximum number of
      zeros of $ \pi_{T(\Fq)}(g)$})-2.
  \end{align*}

  The results follow from the construction  \cite[Section~4.1]{MR2017562}.
  \end{proof}

\begin{theorem}\label{thm:strong} Let $U \subseteq H \subset M$ and let
  $U+U =\{ u_1+u_2 \vert \  u_1, u_2 \in U \}$ be the Minkowski
  sum. Let
  \begin{align*}
    \pi_{T(\Fq)}:\Fq[U+U]&\rightarrow  \Fq^{\vert T(\Fq) \vert}
    \\
    h&\mapsto \pi_{T(\Fq)}(h)=(h(P))_{P\in {T(\Fq)}}.
  \end{align*}

  The linear secret sharing schemes $\mathcal{M}(U)$ of
  Definition \ref{def:LSSS} with $n=(q-1)^r-1$ players, has strong
  multiplication with respect to $\mathcal{A}_{t,n}$ for
  $t \leq t(U)$, where $t(U)$ is the adversary threshold of
  $\mathcal{M}(U)$, if
  \begin{equation}\label{strongbet}
    t \leq n-1 - (\text{the maximal number of zeros of $\pi_{T(\Fq)}(h)$})
  \end{equation}
  for all $h\in \Fq[U+U]$.
\end{theorem}
\begin{proof}
  For $A \in \mathcal{A}_{t,n}$, let
  $B:=T(\Fq) \setminus (\{P_0\} \cup A)$ with $\vert {B} \vert =n-t$
  elements.  For $f, g \in \Fq[U]$, we have that
  $f \cdot g \in \Fq[U+U]$. Consider the linear morphism
  \begin{align}\label{multi}
    \pi_{B}:\Fq[U+U]&\rightarrow  \Fq^{\vert {B} \vert}
    \\
    h&\mapsto(h(P))_{P\in {B}}.
  \end{align}
  evaluating at the points in $B$.

  By assumption $h \in \Fq[U+U]$ can have at most
  $n-t-1 < n-t = \vert {B} \vert$ zeros, therefore $h$ cannot vanish
  identically on $B$, and we conclude that $\pi_B$ is
  injective. Consequently, the products $f(P)\cdot g(P)$ of the shares
  $P \in B$ determine the product of the secrets $f(P_0)\cdot g(P_0)$,
  and the scheme has strong multiplication by Definition \ref{def:defstrong}.
\end{proof}

To determine the product of the secrets from the product of the shares
amounts to decoding the linear code obtained as the image in
\eqref{multi}, see Section \ref{sec:decoding}.

\subsubsection{Hirzebruch surfaces and their associated LSSS.}\label{hirzLSSS}
Let $d, e, r$ be positive integers and let $\square$ be the polytope
in $M_{\mathbb R}$ with vertices $(0,0),(d,0),(d,e+rd),(0,e)$ with refined normal fan  rendered in Figure \ref{fig:hirzpolytop}.
We obtain the following result as a consequence of Theorem
\ref{thm:thresholds} and the bounds obtained in Theorem \ref{thm:saetning4}
on the number of zeros of rational functions on Hirzebruch surfaces.

\begin{theorem}
  Let $\square$ be the polytope in $M_{\mathbb R}$ with vertices
  $(0,0),(d,0),(d,e+rd),(0,e)$. Assume that $d\leq q-2$, $e\leq q-2$
  and that $e+rd\leq q-2$. Let $U= M \cap \square$ be the lattice
  points in $\square$.

  Let $\mathcal{M}(U)$ be the linear secret sharing schemes of Definition
  \ref{def:LSSS} with \emph{support} $T(\mathbb F_q)$ and $(q-1)^2-1$
  players.

  Then the number of lattice points in $\square$ is
  \begin{equation*}
    \vert {U} \vert = \vert(M \cap  \square)\vert =(d+1)(e+1)+r\frac{d(d+1)}{2}.
  \end{equation*}

  The maximal number of zeros of a function $f\in \Fq[U]$ on $T(\Fq)$
  is
  \begin{equation*}
    \max\{d(q-1)+(q-1-d)e,(q-1)(e+dr)\}
  \end{equation*}
  and the reconstruction threshold as defined in Definition
  \ref{def:defthresholds} of $\mathcal{M}(U)$ is
  \begin{equation*}\label{rU}
    r(U) = 1 + \max\{d(q-1)+(q-1-d)e,(q-1)(e+dr)\} .
  \end{equation*}
\end{theorem}
\begin{remark}
  The polytope $-H \setminus -U$ is not convex, so our method using
  intersection theory does not determine the privacy threshold
  $t(U)$. It would be interesting to examine the methods and results
  of \cite{MR2272243}, \cite{MR2476837}, \cite{MR2322944},
  \cite{MR2360532}, \cite{Beelen},\cite{MR3093852} \cite{MR3345095},
\cite{MR3631361} and \cite{MR3705757} and for toric codes in this
  context.
\end{remark}

\subsubsection{Toric surfaces $X_{a,b}$ and  codes $C_{a,b}$ -- LSSS with strong multiplication.}\label{strong}

Let $a, b$ be positive integers $0 \leq b \leq a \leq q-2$, and let
$\square$ be the polytope in $M_{\mathbb R}$ with vertices $(0,0)$,
$(a,0)$, $(b,q-2)$, $(0,q-2)$ rendered in Figure \ref{fig:polytope} and with
normal fan depicted in Figure \ref{fig:cones}. The corresponding toric surfaces and toric codes were studied in Section \ref{X_{a,b}}.

Under these assumptions the polytopes $\square$, $-H\setminus-\square$
and $\square+\square$ are convex and we can use intersection theory on
the associated toric surface to bound the number of zeros of functions
and thresholds.

\begin{theorem} \label{thm:saetning5}
  Assume $a, b$ are integers with $0 \leq b \leq a \leq q-2$.

  Let $\square$ be the polytope in $M_{\mathbb R}$ with vertices
  $(0,0)$, $(a,0)$, $(b,q-2)$, $(0,q-2)$ rendered in Figure
  \ref{fig:polytope}, and let $U= M \cap \square$ be the lattice
  points in $\square$.

  Let $\mathcal{M}(U)$ be the linear secret sharing schemes of Definition
  \ref{def:LSSS} with \emph{support} $T(\mathbb F_q)$ and
  $n=(q-1)^2-1$ players.
  \begin{enumerate}

  \item\label{item:thm5:1} The maximal number of zeros of
    $\pi_{T(\Fq)}(f)$ for $f\in \Fq[U]$ is less than or equal to
    \begin{equation*}
      (q-1)^2-(q-1-a).
    \end{equation*}

  \item\label{item:thm5:2} The reconstruction threshold as defined in Definition
    \ref{def:defthresholds} satisfies
    \begin{equation*}\label{rU2}
      r(U) \leq 1+ (q-1)^2-(q-1-a) .
    \end{equation*}

  \item\label{item:thm5:3} The privacy threshold as defined in Definition
    \ref{def:defthresholds} satisfies
    \begin{equation*}\label{tU2}
      t(U) \geq b-1.
    \end{equation*}

  \item\label{item:thm5:4} Assume $2a\leq q-2$. The secret sharing
    scheme has $t$-strong multiplication for
    \begin{equation*}
      t\leq \min\{b-1,(q-2-2a)-1\}.
    \end{equation*}
  \end{enumerate}
\end{theorem}

\begin{proof}
According to Theorem \ref{thm:saetning3} and Theorem \ref{thm:thresholds}, we have the inequality of
  \ref{item:thm5:2}
  \begin{equation*}
    r(U) \leq 1+ (q-1)^2-(q-1-a).
  \end{equation*}

  We obtain \ref{item:thm5:3} by using the result in \ref{item:thm5:1}
  on the polytope $(q-2,q-2)+(-H \setminus -\square )$ with vertices
  $(0,0)$, $(q-2-b,0)$, $(q-2-a,q-2)$ and $(q-2,q-2)$. The maximum
  number of zeros of $\pi_{T(\Fq)}(g)$ for
  $g\in \Fq[-H \setminus - U]$ is by Lemma \ref{lem:trans} and the result
  in  \ref{item:thm5:1} less than or equal to
  $(q-1)^2-(q-1-(q-2-b))=(q-1)^2-1-b$ and \ref{item:thm5:3} follows
  from  Theorem \ref{thm:thresholds}.

  To prove \ref{item:thm5:4} assume $t \leq (q-2-2a)-1$ and
  $t \leq b-2$. We will use Theorem \ref{thm:strong}.

  Consider the Minkowski sum $U+U$ and let $V=\overline{U + U}$ be its
  reduction modulo $q-1$ as in Corollary \ref{cor:reduction}. Under the
  assumption $2a \leq q-2$, we have that $V=\overline{U + U}$ is the
  lattice points of the integral convex polytope with vertices
  $(0,0)$, $(2a,0)$, $(2b, q-2)$ and $(0,q-2)$.

  By the result in \ref{item:thm5:1} the maximum number of zeros of
  $\pi_{\smash{T(\Fq)}}(h)$ for $h\in \Fq[V]$ is less than or equal to
  $(q-1)^2-(q-1-2a)$. As the number of players is $n=(q-1)^2-1$, the
  right hand side of (\ref{strongbet}) of Theorem
  \ref{thm:strong} is at least $(q-2-2a)-1$, which by assumption is
  at least $t$.

  By assumption $t \leq b-1$ and from \ref{item:thm5:3} we have that
  $b-1 \leq t(U)$. We conclude that $t \leq t(U)$.
\end{proof}

The bound in (i) of Theorem \ref{thm:saetning5} can be improved for $q$ sufficiently large. The polytope $\square$ includes the rectangle with vertices $(0, 0), (b, 0), (b, q -2), (0, q -2)$, which is the maximal Minkowski-reducible sub-polygon. The results of \cite{MR2272243} and \cite{MR2476837} give that  the maximal number of zeroes of $\pi_{T(\Fq)}(f)$ for $f\in \Fq[U]$ cannot be larger than $(q-1)^2-(q-1)+b$, with \cite{MR2476837} giving the best bounds on $q$.

\section{Asymmetric Quantum Codes on Toric Surfaces}\label{QC}

A source on Quantum Computation and Quantum Information is \cite{Nielsen00}.
\subsection{Introduction}

Our construction in Section \ref{sec:toriccodes}  of toric codes is suitable for constructing
quantum codes by the method of A. R. Calderbank, P. W. Shor and A. M. Steane -- the (CSS) method. Our constructions
extended similar results obtained by A. Ashikhmin, S. Litsyn and
M.A. Tsfasman in \cite{tsfasman} from Goppa codes on algebraic curves.

Works of
\cite{PhysRevA.52.R2493} and

\cite{MR1421749}, \cite{MR1398854} initiated the study and
construction of quantum error-correcting codes.

\cite{Calderbank19961098},

\cite{MR1450603} and

\cite{Steane19992492} produced stabilizer codes (CSS) from linear
codes containing their dual codes. For details see for example
\cite{959288}, \cite{MR1665774} and
\cite{MR1750540}.

Asymmetric quantum error-correcting codes are quantum codes defined
over biased quantum channels: qubit-flip and phase-shift errors may
have equal or different probabilities. The code construction is the
CSS construction based on two linear codes. The construction appeared
originally in \cite{2007arXiv0709.3875E}, \cite{PhysRevA.75.032345}
and \cite{PhysRevA.77.062335}. We present new families of toric
surfaces, toric codes and associated asymmetric quantum
error-correcting codes.

In \cite{MR3015727}
 results on toric quantum codes are obtained by constructing a dualizing differential form for the toric surfaces.

A different approach to quantum codes defined on  toric surfaces was originally presented in \cite{Kitaev_1997}. In \cite{PhysRevX.2.021004} and \cite{MR2648534} different codes are constructed by this method.

\subsection{Asymmetric Quantum Codes}

Let $\mathcal{H}$ be the Hilbert space
$\mathcal{H}=\mathbb C^{q^n}=\mathbb C^q \otimes \mathbb C^q \otimes
\dots \otimes \mathbb C^q$.  Let $\ket{x}, x \in \mathbb F_q$ be an
orthonormal basis for $\mathbb C^q$.  For $a,b \in \mathbb F_q$, the
unitary operators $X(a)$ and $Z(b)$ in $\mathbb C^q$ are
\begin{equation}
  X(a)\ket{x}=\ket{x+a},\qquad
  Z(b)\ket{x}=\omega^{\tr(bx)}\ket{x},
\end{equation}
where $\omega=\exp(2\pi i/p)$ is a primitive $p$'th root of unity and
$\tr$ is the trace operation from $\mathbb F_q$ to $\mathbb F_p$.

For ${\bf a}=(a_1,\dots, a_n)\in \mathbb F_q^n$ and
${\bf b}=(b_1,\dots, b_n)\in \mathbb F_q^n$
\begin{align*}
  X({\bf a}) &= X(a_1)\otimes \dots \otimes X(a_n)
  \\
  Z({\bf b}) &= Z(b_1)\otimes\ \dots \otimes Z(b_n)
\end{align*}
are the tensor products of $n$ error operators.

With
\begin{align*}
  {\mathbf E}_x&=\{ X({\bf a})=\bigotimes_{i=1}^n X(a_i) \vert \  {\bf a}
    \in \mathbb F_q^n, a_i \in \mathbb F_q \},
  \\
  {\mathbf E}_z&=\{ Z({\bf b})=\bigotimes_{i=1}^n Z(b_i) \vert \  {\bf b}
    \in \mathbb F_q^n, b_i \in \mathbb F_q \}
\end{align*}
the error groups $\mathbf{G}_x$ and $\mathbf{G}_z$ are
\begin{align*}
  \mathbf{G}_x &=\{ \omega^{c}{\bf E}_x=\omega^{c}X({\bf a}) \vert \  {\bf a} \in
    \mathbb F_q^n, c\in \mathbb F_p \},\nonumber
  \\
  \mathbf{G}_z &= \{ \omega^{c}{\bf E}_z=\omega^{c}Z({\bf b}) \vert \  {\bf b}
    \in \mathbb F_q^n, c \in \mathbb F_p \}.
\end{align*}

It is assumed that the groups $\bf G_x$ and $\bf G_z$ represent the
qubit-flip and phase-shift errors.

\begin{definition}(Asymmetric quantum code).
  A $q$-ary asymmetric quantum code $Q$, denoted by
  $[[n,k,d_z/d_x]]_q$, is a $q^k$ dimensional subspace of the Hilbert
  space $\mathbb{C}^{q^n}$ and can control all bit-flip errors up to
  $\lfloor \frac{d_x-1}{2}\rfloor$ and all phase-flip errors up to
  $\lfloor \frac{d_z-1}{2}\rfloor$. The code $Q$ detects $(d_x-1)$
  qubit-flip errors as well as detects $(d_z-1)$ phase-shift errors.
\end{definition}

Let $C_1$ and $C_2$ be two linear error-correcting codes over $\mathbb F_q$, and let $[n,k_1,d_1]_q$ and
$[n,k_2,d_2]_q$ be their parameters.  For the dual codes
$C_i^{\perp}$, we have $\dim{C_i^{\perp}}=n-k_i$ and if
$C_{1}^\perp \subseteq C_{2}$ then $C_{2}^\perp \subseteq C_1$.

\begin{lemma}\label{lem:lin/qua}
  Let $C_i$ for $i=1,2$ be linear error-correcting codes with
  parameters $[n,k_i,d_i]_q$ such that $C_1^\perp \subseteq C_{2}$ and
  $C_2^\perp \subseteq C_{1}$.  Let
  $d_x= \min\{\wt(C_{1} \setminus C_2^\perp), \wt(C_{2} \setminus C_{1
  }^\perp)\}$, and
  $d_z= \max\{\wt(C_{1} \setminus C_2^\perp), \wt(C_{2} \setminus C_{1
  }^\perp)\}$. Then there is an asymmetric quantum code with
  parameters $[[n,k_1+k_2-n,d_z/d_x]]_q$. The quantum code is pure to
  its minimum distance, meaning that if
  $\wt(C_1)=\wt(C_1\setminus C_2^\perp)$, then the code is pure to
  $d_x$, also if $\wt(C_2)=\wt(C_2\setminus C_1^\perp)$, then the
  code is pure to $d_z$.
\end{lemma}

This construction is well-known, see for example \cite{959288},
\cite{MR1665774}, \cite{PhysRevA.52.R2493}, \cite{MR1421749},
\cite{PhysRevA.54.4741} , \cite{771249} \cite{5503199}. The error
groups $\bf G_x$ and $\bf G_z$ can be mapped to the linear codes $C_1$
and $C_2$.

\subsubsection{Asymmetric Quantum Codes from Toric Codes}

Let $a, b$ be positive integers $0 \leq b \leq a \leq q-2$, and let
$\square_{a,b}$ be the polytope in $M_{\mathbb R}$ with vertices $(0,0)$,
$(a,0)$, $(b,q-2)$, $(0,q-2)$. The corresponding toric surfaces and toric codes were studied in Section \ref{X_{a,b}}.

From the results in Section \ref{sec:dualtoric} we conclude that the dual
code $C_{a,b}^{\perp}$ is the toric code associated to the polytope $\square_{a^{\perp},b^{\perp}}$ 
with vertices $(0,0)$,
$(a^{\perp},0)$, $(b^{\perp},q-2)$, $(0,q-2)$
where $b^{\perp}=q-2-a$ and $a^{\perp}=q-2-b$.

 For $a_1 \leq a_2$ we have the inclusions of polytopes $\square_{a_2^{\perp},b_2^{\perp}} \subseteq \square_{a_1^{\perp},b_1^{\perp}}$, see Figure \ref{fig:inklu},
and corresponding inclusions of the associated toric codes
\begin{equation*}
  C_{a_2,b_2}^{\perp}= \subseteq C_{a_1, b_1}, \quad C_{a_1,b_1}^{\perp} \subseteq C_{a_2,b_2}.
\end{equation*}

The nested codes gives by the construction of Lemma  \ref{lem:lin/qua} and
the discussion above rise to an asymmetric quantum code $Q_{b_1,b_2}$.

\begin{theorem}[(Asymmetric quantum codes $Q_{a_1, b_1, a_2, b_2}$)]

Let $\mathbb F_q$ be a field with $q$ elements and let $0\leq b_i\leq a_i \leq q-2$ for $i=1,2$.

  Then there is an asymmetric quantum code $Q_{a_1, a_2, b_1,b_2}$ with
  parameters
  $[[(q-1)^2, k_1+k_2-(q-1)^2,d_z/d_x]]_q$, where
  \begin{align*}
  k_i=&\frac{(q-1)(a_i+b_i+1)+\gcd(a_i-b_i,q-2)+1}{2}\ ,\quad  i=1,2\ ,\\ 
   d_z \geq& q-1-a_1\ ,
    \\
    d_x \geq& q-1-a_2\ .
  \end{align*}

\end{theorem}
\begin{proof}
  The parameters and claims follow directly from Lemma  \ref{lem:lin/qua}
  and Theorem \ref{thm:saetning3}.
\end{proof}
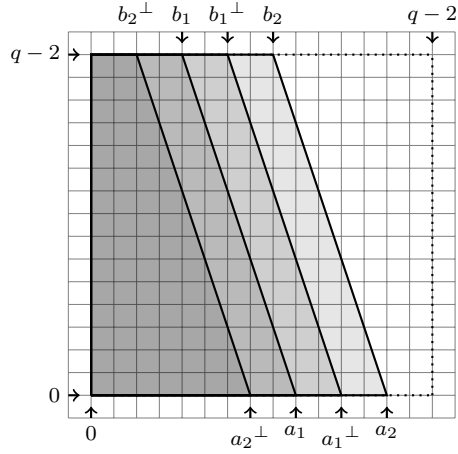
\begin{figure}
  \centering
  \begin{tikzpicture}[scale=0.3]
      \draw[step=1cm,gray,very thin] (-1,-1) grid (16,16);
      \draw[thick,dotted](0,15)--(15,15);
      \draw[thick,dotted](15,0)--(15,15);
      \draw[thick,dotted](0,0)--(15,0);
      \filldraw[thick,fill opacity=0.1 ] (0,0)--(13,0)--(8,15)--(0,15)--(0,0);
      \filldraw[thick,fill opacity=0.1 ] (0,0)--(11,0)--(6,15)--(0,15)--(0,0);
      \filldraw[thick,fill opacity=0.1 ] (0,0)--(9,0)--(4,15)--(0,15)--(0,0);
      \filldraw[thick,fill opacity=0.1 ] (0,0)--(7,0)--(2,15)--(0,15)--(0,0);

      \draw[thick,<-] (9,-0.5)--(9,-1);
      \draw[thick,<-] (0,-0.5)--(0,-1);
      \draw[thick,<-] (7,-0.5)--(7,-1);
      \draw[thick,<-] (13,-0.5)--(13,-1);
      \draw[thick,<-] (11,-0.5)--(11,-1);
      \draw[thick,->] (4,16)--(4,15.5);
      \draw[thick,->] (6,16)--(6,15.5);
      \draw[thick,->] (8,16)--(8,15.5);
      \draw[thick,->] (15,16)--(15,15.5);
      \draw[thick,->] (-1,15)--(-0.5,15);
      \draw[thick,->] (-1,0)--(-0.5,0);
      \draw (-1,15) node[anchor=east] {$q-2$};
      \draw (-1,0) node[anchor=east] {$0$};
      \draw (4,16) node[anchor=south] {$b_1$};
      \draw (2,16) node[anchor=south] {${b_2}^{\perp}$};
      \draw (8,16) node[anchor=south] {$b_2$};
      \draw (6,16) node[anchor=south] {${b_1}^{\perp}$};
      \draw (0,-1) node[anchor=north] {$0$};
      \draw (9,-1) node[anchor=north] {$a_1$};
      \draw (13,-1) node[anchor=north] {$a_2$};
      \draw (7,-1) node[anchor=north] {${a_2}^{\perp}$};
      \draw (11,-1) node[anchor=north] {${a_1}^{\perp}$};
      \draw (15,16) node[anchor=south] {$q-2$};
    \end{tikzpicture}
    \caption{The polytope $\square_{a_i,b_i}$ is the polytope with
      vertices $(0,0),(a_i,0),(b_i,q-2),
      (0,q-2)$. The polytopes giving the dual toric codes have
      vertices $(0,0), (a_i^{\perp}=q-2-b_i,0), (b_i^{\perp}=q-2-a_i,q-2), (0,q-2)$.}\label{fig:inklu}
\end{figure}
\section{Toric Codes, Multiplicative Structure and Decoding}
\label{sec:decoding}

We utilized the inherent multiplicative structure on toric codes
to \emph{decode} toric codes, resembling the decoding of
Reed-Solomon codes and decoding by error correcting pairs, see
\cite{MR1181934} ,

\cite{RK}
and
\cite{MR3502016}.

\subsection{Multiplicative structure}
In the notation of Section \ref{sec:toriccodes} let $\square$ and
$\tilde{\square}$ be polyhedra in $\mathbb R^r$ and let
$\square +\tilde{\square}$ denote their Minkowski sum. Let
$U=\square \cap \mathbb Z^r$ and
$\tilde{U}=\tilde{\square} \cap \mathbb Z^r$.  The map
\begin{align*}
  \Fq[U] \oplus \Fq[\tilde{U}] &\rightarrow \Fq[U+\tilde{U}]
  \\
  (f,g)&\mapsto f \cdot g.
\end{align*}
induces a multiplication on the associated toric codes
\begin{align*}
  C_{\square} \oplus C_{\tilde{\square}} &\rightarrow 
  C_{\square+\tilde{\square}}
  \\
  (c,\tilde{c})&\mapsto c \star \tilde{c}
\end{align*}
with coordinatewise multiplication of the codewords -- the
\emph{Schur} product.

Our goal is to use the multiplicative structure to correct $t$ errors
on the toric code $C_{\square}$. This is achieved choosing another
toric code $C_{\tilde{\square}}$ that helps to reduce error-correcting
to a \emph{linear} problem.

Assume from now on:
\begin{enumerate}
\item\label{item:ass:1} $\vert {\tilde{U}}\vert  >t$, where
  $\tilde{U}=\tilde{\square} \cap \mathbb Z^2$.
\item\label{item:ass:2} $d(C_{\square+\tilde{\square}}) > t$, where
  $d(C_{\square+\tilde{\square}})$ is the minimum distance of
  $C_{\square+\tilde{\square}}$.
\item\label{item:ass:3} $d(C_{\tilde{\square}}) > n-d(C_{\square})$,
  where $d(C_{\square})$ and $d(C_{\tilde{\square}})$ are the minimum
  distances of $C_{\square}$ and $C_{\tilde{\square}}$.
\end{enumerate}

\subsection{Error-locating}
Let the received word be $y(P)=f(P)+e(P)$ for $P\in T(\Fq)$, with
$f \in \Fq[U]$ and error $e$ of Hamming-weight at most $t$ with
support $T\subseteq T(\Fq)$, such that $\vert \ {T} \ \vert \leq t$.

From \ref{item:ass:1}, it follows that there is a
$g \in \Fq[\tilde{U}]$, such that $g_{\vert T}=0$ -- an
\emph{error-locator}.  To find $g$, consider the linear map:
\begin{align}
  \Fq[\tilde{U}] \oplus \Fq[U+\tilde{U}] &\rightarrow 
  \Fq^n\label{map}
  \\
  (g,h)&\mapsto \bigl( g(P)y(P)-h(P)\bigr) _{P\in T(\Fq)}.
\end{align}

As $y(P)-f(P)=0$ for $P \notin T$ (recall that the support of the
error $e$ is $T\subseteq T(\Fq)$, we have that $g(P)y(P)-(g\cdot f )(P)=0$ for all
$P \in T(\Fq)$. That is $(g,h=g\cdot f)$ is in the kernel of
\eqref{map}.

\begin{lemma}
  Let $(g,h)$ be in the kernel of \eqref{map}. Then $g\vert T
  =0$ and
  $h=g\cdot f$.
\end{lemma}
\begin{proof}
  \begin{equation}
    e(P)=y(P)-f(P), \quad P \in T(\Fq).
  \end{equation}
  Coordinate wise multiplication yields by \eqref{map}
  \begin{align*}
    g(P)e(P)&=g(P)y(P)-g(P)f(P)
    \\
    &=h(P)-g(P)f(P)
  \end{align*}
  for $P\in T(\Fq)$.  The left hand side has Hamming weight at most
  $t$, the right hand side is a code word in
  $C_{\square+\tilde{\square}}$ with minimal distance strictly larger
  than $t$ by assumption \ref{item:ass:2}. Therefore both sides
  equal~$0$.
\end{proof}

\subsection{Error-correcting}
\begin{lemma}
  Let $(g,h)$ be in the kernel of \eqref{map} with
  $g\vert T =0$ and $g \neq
  0$. There is a unique $f$ such that $h=g\cdot f$.
\end{lemma}
\begin{proof}
  As in the above proof, we have
  \begin{equation}
    g(P)y(P)-g(P)f(P)=0, \quad P\in T(\Fq).
  \end{equation}
  Let $Z(g)$ be the zero-set of $g$ with $T\subseteq Z(g)$.  For
  $P \notin Z(g)$, we have $y(P)=f(P)$ and there are at least
  $d(C_{\tilde{\square}}) > n-d(C_{\square})$ such points by
  \ref{item:ass:3}. This determines $f$ uniquely as it is determined
  by the values in $n-d(C_{\square})$ points.
\end{proof}

\begin{example}
  Let $\square$ be the convex polytope with vertices $(0,0)$, $(a,0)$
  and $(0,a)$. Let $\tilde{\square}$ be the convex polytope with
  vertices $(0,0)$, $(b,0)$ and $(0,b)$. Their Minkowski sum
  $\square +\tilde{\square}$ is the convex polytope with vertices
  $(0,0)$, $(a+b,0)$ and $(0,a+b)$, see  Figure \ref{fig:polytope2}.

   We have
  that $n=(q-1)^2$,
  $\vert \ {\tilde{\square}}\ \vert = \frac{(b+1)\*(b+2)}{2}$,
  $d(C_{\square})=(q-1)(q-1-a)$, $d(C_{\tilde{\square}})=(q-1)(q-1-b)$
  and $ d(C_{\square+\tilde{\square}})=(q-1)(q-1-(a+b)) $ for the
  associated codes over $\mathbb F_q$, see \cite{6ffeb030f4f511dd8f9a000ea68e967b},
\cite{39bd8e90f4f211dd8f9a000ea68e967b} and
\cite{53ee7c6020b511dcbee902004c4f4f50}.

  Let $q=16, a=4$ and $b=8$. Then $n=225$,
  $\vert \ {\tilde{\square}}\ \vert=45$, $ d(C_{\square})=165$,
  $d(C_{\tilde{\square}})=105$ and
  $ d(C_{\square+\tilde{\square}})=45 $.

  As $d(C_{\tilde{\square}})=105>60=n-d(C_{\square})$, the procedure
  corrects $t$ errors with
  \begin{equation*}t < \min \{d(C_{\square+\tilde{\square}}),
  \vert \ {\tilde{\square}}\ \vert \}=45\ .
\end{equation*}
  \begin{figure}
    \centering
    \begin{tikzpicture}[scale=0.35]

        \draw[very thick] (0,0)--(14,0);
        \draw[very thick] (0,0)--(0,14);
        \draw[thick, fill=black!5] (0,0)--(10,0)--(0,10)--(0,0);
        \draw[thick,fill=black!10] (0,0)--(12,0)--(0,12)--(0,0);
        \draw (12,0) node[anchor=north]{$a+b$};
        \draw (0,12) node[anchor=east]{$a+b$};
        \draw[thick,fill=black!10] (0,0)--(8,0)--(0,8)--(0,0);
        \draw[thick,fill=black!15] (0,0)--(4,0)--(0,4)--(0,0);
        \draw (8,0) node[anchor=north]{$b$};
        \draw (0,8) node[anchor=east]{$b$};
        \draw (4,0) node[anchor=north]{$a$};
        \draw (0,4) node[anchor=east]{$a$};
        \draw (4,0) node[anchor=north]{$a$};
        \draw (0,4) node[anchor=east]{$a$};
        \draw (14,0) node[anchor=west]{$q-2$};
        \draw (0,14) node[anchor=east]{$q-2$};
        \draw[step=1cm,gray,very thin,dashed]
        (0,0) grid (14,14);
      \end{tikzpicture}
    \caption{The convex polytope ${\square}$ with vertices
      $(0,0), (a,0)$ and $(0,a)$. The convex polytope
      ${\tilde{\square}}$ with vertices $(0,0), (b,0)$ and
      $(0,b))$. Their Minkowski sum $\square+\tilde{\square}$ having
      vertices $(0,0), (a+b,0)$ and $(0,a+b)$.}\label{fig:polytope2}
  \end{figure}
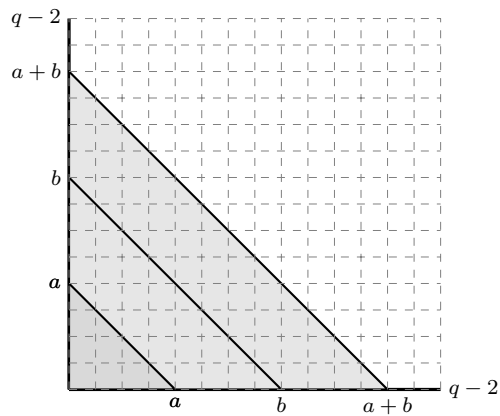
\end{example}
\begin{remark} \emph{Error correcting pairs}.
  \cite{MR1181934} and
  \cite{RK} introduced the concept of error-correcting pairs for a
  linear code, see also
  \cite{MR3502016}. Specifically for a linear code $C \subseteq \mathbb F_q^n$
  a $t$-error correcting pair consists of two linear codes
  $A,B \subseteq \mathbb F_q^n$, such that
  \begin{equation}
    (A\star B) \perp C,\quad \dim_{\mathbb F_q} A > t,\quad d(B^\perp)>t,\quad d(A)+d(C)>n.
  \end{equation}
  Here $A\star B =\{ a \star b \vert \  a \in A, b \in B \}$ and
  $\perp$ denotes orthogonality with respect to the usual inner
  product. They described the known decoding algorithms for decoding
  $t$ or fewer errors in this framework.

  Also the decoding in the present paper can be described in this
  framework, taking $C=C_{\square}, A=C_{\tilde{\square}}$ and
  $B=(C\star A)^{\perp}$ using Proposition \ref{prop:orto}.
\end{remark}

We utilized the inherent multiplicative structure on toric codes
to \emph{decode} toric codes, resembling the decoding of
Reed-Solomon codes and decoding by error correcting pairs, see

\cite{MR1181934} ,
\cite{RK} and

\cite{MR3502016}.

\begin{acknowledgements} I would like to thank the anonymous reviewers for their valuable remarks and constructive comments, which significantly contributed to the quality of a paper. 
\end{acknowledgements}

\bibliographystyle{alpha} 
\bibliography{SSS}

\newcommand{\etalchar}[1]{$^{#1}$}
\begin{thebibliography}{BAO{\etalchar{+}}12}

\bibitem[AA10]{5503199}
S.~A. Aly and A.~Ashikhmin.
\newblock Nonbinary quantum cyclic and subsystem codes over
  asymmetrically-decohered quantum channels.
\newblock In {\em 2010 IEEE Information Theory Workshop on Information Theory
  (ITW 2010, Cairo)}, pages 1--5, Jan 2010.

\bibitem[AK01]{959288}
A.~Ashikhmin and E.~Knill.
\newblock Nonbinary quantum stabilizer codes.
\newblock {\em IEEE Transactions on Information Theory}, 47(7):3065--3072, Nov
  2001.

\bibitem[ALT01]{tsfasman}
A.~Ashikhmin, S.~Litsyn, and M.A. Tsfasman.
\newblock Asymptotically good quantum codes.
\newblock {\em Physical Review A -- Atomic, Molecular, and Optical Physics},
  63(3):1--5, 2001.

\bibitem[APS10]{MR2648534}
Clarice~Dias Albuquerque, Reginaldo Palazzo, Jr., and Eduardo~Brandani Silva.
\newblock Construction of new toric quantum codes.
\newblock In {\em Finite fields: theory and applications}, volume 518 of {\em
  Contemp. Math.}, pages 1--10. Amer. Math. Soc., Providence, RI, 2010.

\bibitem[BAO08]{MR2377234}
Maria Bras-Amor{\'o}s and Michael~E. O'Sullivan.
\newblock Duality for some families of correction capability optimized
  evaluation codes.
\newblock {\em Adv. Math. Commun.}, 2(1):15--33, 2008.

\bibitem[BAO{\etalchar{+}}12]{PhysRevX.2.021004}
H.~Bombin, Ruben~S. Andrist, Masayuki Ohzeki, Helmut~G. Katzgraber, and M.~A.
  Martin-Delgado.
\newblock Strong resilience of topological codes to depolarization.
\newblock {\em Phys. Rev. X}, 2:021004, Apr 2012.

\bibitem[BBGS14]{MR3224833}
Alp Bassa, Peter Beelen, Arnaldo Garcia, and Henning Stichtenoth.
\newblock Galois towers over non-prime finite fields.
\newblock {\em Acta Arith.}, 164(2):163--179, 2014.

\bibitem[BGW88]{DBLP:conf/stoc/Ben-OrGW88}
Michael Ben{-}Or, Shafi Goldwasser, and Avi Wigderson.
\newblock Completeness theorems for non-cryptographic fault-tolerant
  distributed computation (extended abstract).
\newblock In Simon \cite{DBLP:conf/stoc/STOC20}, pages 1--10.

\bibitem[Bla79]{Blakley1979}
G.R. Blakley.
\newblock Safeguarding cryptographic keys.
\newblock In {\em Proceedings of the 1979 AFIPS National Computer Conference},
  pages 313--317, Monval, NJ, USA, 1979. AFIPS Press.

\bibitem[BR09]{Beelen}
Peter Beelen and Diego Ruano.
\newblock The order bound for toric codes.
\newblock In Maria Bras-Amorós and Tom H\o{}holdt, editors, {\em Applied
  Algebra, Algebraic Algorithms and Error-Correcting Codes}, volume 5527 of
  {\em Lecture Notes in Computer Science}, pages 1--10. Springer Berlin
  Heidelberg, 2009.

\bibitem[Cas10]{17425}
I.~Cascudo.
\newblock {\em On {Asymptotically} {Good} {Strongly} {Multiplicative} {Linear}
  {Secret} {Sharing}}.
\newblock PhD thesis, University of Oviedo, 2010.

\bibitem[CC06]{CC}
Hao Chen and Ronald Cramer.
\newblock Algebraic geometric secret sharing schemes and secure multi-party
  computations over small fields.
\newblock In Cynthia Dwork, editor, {\em Advances in Cryptology - CRYPTO 2006},
  volume 4117 of {\em Lecture Notes in Computer Science}, pages 521--536.
  Springer Berlin Heidelberg, 2006.

\bibitem[CCD88]{DBLP:conf/stoc/ChaumCD88}
David Chaum, Claude Cr{\'e}peau, and Ivan Damg{\aa}rd.
\newblock Multiparty unconditionally secure protocols (extended abstract).
\newblock In Simon \cite{DBLP:conf/stoc/STOC20}, pages 11--19.

\bibitem[CCG{\etalchar{+}}07]{MR2449216}
Hao Chen, Ronald Cramer, Shafi Goldwasser, Robbert de~Haan, and Vinod
  Vaikuntanathan.
\newblock Secure computation from random error correcting codes.
\newblock In {\em Advances in cryptology---{EUROCRYPT} 2007}, volume 4515 of
  {\em Lecture Notes in Comput. Sci.}, pages 291--310. Springer, Berlin, 2007.

\bibitem[CDM00]{CDM}
Ronald Cramer, Ivan Damg\aa{}rd, and Ueli Maurer.
\newblock General secure multi-party computation from any linear secret-sharing
  scheme.
\newblock In Bart Preneel, editor, {\em Advances in Cryptology — EUROCRYPT
  2000}, volume 1807 of {\em Lecture Notes in Computer Science}, pages
  316--334. Springer Berlin Heidelberg, 2000.

\bibitem[CDN15]{DBLP:books/cu/CDN2015}
Ronald Cramer, Ivan Damg{\aa}rd, and Jesper~Buus Nielsen.
\newblock {\em Secure Multiparty Computation and Secret Sharing}.
\newblock Cambridge University Press, 2015.

\bibitem[CLS11]{1223.14001}
David~A. Cox, John~B. Little, and Henry~K. Schenck.
\newblock {\em {Toric varieties.}}
\newblock {Graduate Studies in Mathematics 124. Providence, RI: American
  Mathematical Society (AMS). xxiv, 841~p. }, 2011.

\bibitem[CRSS98]{MR1665774}
A.~Robert Calderbank, Eric~M. Rains, P.~W. Shor, and Neil J.~A. Sloane.
\newblock Quantum error correction via codes over {$\mathrm{GF}(4)$}.
\newblock {\em IEEE Trans. Inform. Theory}, 44(4):1369--1387, 1998.

\bibitem[CS96]{Calderbank19961098}
A.R. Calderbank and P.W. Shor.
\newblock Good quantum error-correcting codes exist.
\newblock {\em Physical Review A -- Atomic, Molecular, and Optical Physics},
  54(2):1098--1105, 1996.

\bibitem[ESCH07]{2007arXiv0709.3875E}
Z.~W.~E. {Evans}, A.~M. {Stephens}, J.~H. {Cole}, and L.~C.~L. {Hollenberg}.
\newblock {Error correction optimisation in the presence of X/Z asymmetry}.
\newblock {\em ArXiv e-prints}, September 2007.

\bibitem[Ful93a]{Fulton:1436535}
William Fulton.
\newblock {\em {Introduction to toric varieties}}.
\newblock Annals of mathematics studies. Princeton Univ. Press, Princeton, NJ,
  1993.

\bibitem[Ful93b]{MR1234037}
William Fulton.
\newblock {\em Introduction to toric varieties}, volume 131 of {\em Annals of
  Mathematics Studies}.
\newblock Princeton University Press, Princeton, NJ, 1993.
\newblock The William H. Roever Lectures in Geometry.

\bibitem[Ful98]{MR1644323}
William Fulton.
\newblock {\em Intersection theory}, volume~2 of {\em Ergebnisse der Mathematik
  und ihrer Grenzgebiete. 3. Folge. A Series of Modern Surveys in Mathematics
  [Results in Mathematics and Related Areas. 3rd Series. A Series of Modern
  Surveys in Mathematics]}.
\newblock Springer-Verlag, Berlin, second edition, 1998.

\bibitem[Gop88]{MR1029027}
V.~D. Goppa.
\newblock {\em Geometry and codes}, volume~24 of {\em Mathematics and its
  Applications (Soviet Series)}.
\newblock Kluwer Academic Publishers Group, Dordrecht, 1988.
\newblock Translated from the Russian by N. G. Shartse.

\bibitem[Han92]{MR1186416}
Johan~P. Hansen.
\newblock Deligne-{L}usztig varieties and group codes.
\newblock In {\em Coding theory and algebraic geometry ({L}uminy, 1991)},
  volume 1518 of {\em Lecture Notes in Math.}, pages 63--81. Springer, Berlin,
  1992.

\bibitem[Han98]{6ffeb030f4f511dd8f9a000ea68e967b}
{Johan P.} Hansen.
\newblock Toric surfaces and codes.
\newblock In {\em Information Theory Workshop}, pages 42--43. IEEE, 1998.

\bibitem[Han00]{39bd8e90f4f211dd8f9a000ea68e967b}
{Johan P.} Hansen.
\newblock Toric surfaces and error-correcting codes.
\newblock In J.~Buchmann, T.~Hoeholdt, H.~Stichtenoth, and H.~Tapia-Recillas,
  editors, {\em Coding theory, cryptography and related areas}, pages 132--142.
  Springer, 2000.

\bibitem[Han01]{Hansen2001530}
Søren~Have Hansen.
\newblock Error-correcting codes from higher-dimensional varieties.
\newblock {\em Finite Fields and Their Applications}, 7(4):530 -- 552, 2001.

\bibitem[Han02]{53ee7c6020b511dcbee902004c4f4f50}
{Johan P.} Hansen.
\newblock Toric varieties {H}irzebruch surfaces and error-correcting codes.
\newblock {\em Applicable Algebra in Engineering, Communication and Computing},
  13(4):289--300, 2002.

\bibitem[Han13]{MR3015727}
Johan~P. Hansen.
\newblock Quantum codes from toric surfaces.
\newblock {\em IEEE Trans. Inform. Theory}, 59(2):1188--1192, 2013.

\bibitem[HP93]{MR1225959}
Johan~P. Hansen and Jens~Peter Pedersen.
\newblock Automorphism groups of {R}ee type, {D}eligne-{L}usztig curves and
  function fields.
\newblock {\em J. Reine Angew. Math.}, 440:99--109, 1993.

\bibitem[HS90]{MR1325513}
Johan~P. Hansen and Henning Stichtenoth.
\newblock Group codes on certain algebraic curves with many rational points.
\newblock {\em Appl. Algebra Engrg. Comm. Comput.}, 1(1):67--77, 1990.

\bibitem[IM07]{PhysRevA.75.032345}
Lev Ioffe and Marc M\'ezard.
\newblock Asymmetric quantum error-correcting codes.
\newblock {\em Phys. Rev. A}, 75:032345, Mar 2007.

\bibitem[JH17]{tom}
Jom Justesen and Tom H\o{}holdt.
\newblock {\em A Course in Error-Correcting Codes (EMS Textbooks in
  Mathematics)}.
\newblock European Mathematical Society, 2017.

\bibitem[Kit97]{Kitaev_1997}
A~Yu Kitaev.
\newblock Quantum computations: algorithms and error correction.
\newblock {\em Russian Mathematical Surveys}, 52(6):1191--1249, dec 1997.

\bibitem[K{\"o}t92]{RK}
Ralf K{\"o}tter.
\newblock A unified description of an error locating procedure for linear
  codes.
\newblock In {\em Proceedings of Algebraic and Combinatorial Coding Theory},
  pages 113--117. Voneshta Voda, 1992.

\bibitem[Lit13]{MR3093852}
John~B. Little.
\newblock Remarks on generalized toric codes.
\newblock {\em Finite Fields Appl.}, 24:1--14, 2013.

\bibitem[Lit17a]{MR3705757}
John~B. Little.
\newblock Corrigendum to ``{T}oric codes and finite geometries'' [{F}inite
  {F}ields {A}ppl. 45 (2017) 203--216] [ {MR}3631361].
\newblock {\em Finite Fields Appl.}, 48:447--448, 2017.

\bibitem[Lit17b]{MR3631361}
John~B. Little.
\newblock Toric codes and finite geometries.
\newblock {\em Finite Fields Appl.}, 45:203--216, 2017.

\bibitem[LS06]{MR2272243}
John Little and Hal Schenck.
\newblock Toric surface codes and {M}inkowski sums.
\newblock {\em SIAM J. Discrete Math.}, 20(4):999--1014 (electronic), 2006.

\bibitem[LS07]{MR2322944}
John Little and Ryan Schwarz.
\newblock On toric codes and multivariate {V}andermonde matrices.
\newblock {\em Appl. Algebra Engrg. Comm. Comput.}, 18(4):349--367, 2007.

\bibitem[Mas01]{MR2017562}
James~L. Massey.
\newblock Some applications of code duality in cryptography.
\newblock {\em Mat. Contemp.}, 21:187--209, 2001.
\newblock 16th School of Algebra, Part II (Portuguese) (Bras{\'{\i}}lia, 2000).

\bibitem[MCP16]{MR3502016}
Irene M\'arquez-Corbella and Ruud Pellikaan.
\newblock A characterization of {MDS} codes that have an error correcting pair.
\newblock {\em Finite Fields Appl.}, 40:224--245, 2016.

\bibitem[MS77]{MR0465510}
F.~J. MacWilliams and N.~J.~A. Sloane.
\newblock {\em The theory of error-correcting codes. {II}}.
\newblock North-Holland Publishing Co., Amsterdam-New York-Oxford, 1977.
\newblock North-Holland Mathematical Library, Vol. 16.

\bibitem[NC00]{Nielsen00}
Michael~A. Nielsen and Isaac~L. Chuang.
\newblock {\em Quantum Computation and Quantum Information}.
\newblock Cambridge University Press, 2000.

\bibitem[Oda88]{Oda1988}
Tadao Oda.
\newblock {\em Convex bodies and algebraic geometry}.
\newblock Springer, 1988.

\bibitem[Pel92]{MR1181934}
Ruud Pellikaan.
\newblock On decoding by error location and dependent sets of error positions.
\newblock {\em Discrete Math.}, 106/107:369--381, 1992.
\newblock A collection of contributions in honour of Jack van Lint.

\bibitem[Rua07]{MR2360532}
Diego Ruano.
\newblock On the parameters of {$r$}-dimensional toric codes.
\newblock {\em Finite Fields Appl.}, 13(4):962--976, 2007.

\bibitem[Rua09]{MR2499927}
Diego Ruano.
\newblock On the structure of generalized toric codes.
\newblock {\em J. Symbolic Comput.}, 44(5):499--506, 2009.

\bibitem[SEDH08]{PhysRevA.77.062335}
Ashley~M. Stephens, Zachary W.~E. Evans, Simon~J. Devitt, and Lloyd C.~L.
  Hollenberg.
\newblock Asymmetric quantum error correction via code conversion.
\newblock {\em Phys. Rev. A}, 77:062335, Jun 2008.

\bibitem[Sha79]{journals/cacm/Shamir79}
Adi Shamir.
\newblock How to share a secret.
\newblock {\em Commun. ACM}, 22(11):612--613, 1979.

\bibitem[Sho95]{PhysRevA.52.R2493}
Peter~W. Shor.
\newblock Scheme for reducing decoherence in quantum computer memory.
\newblock {\em Phys. Rev. A}, 52:R2493--R2496, Oct 1995.

\bibitem[Sho96]{MR1450603}
Peter~W. Shor.
\newblock Fault-tolerant quantum computation.
\newblock In {\em 37th {A}nnual {S}ymposium on {F}oundations of {C}omputer
  {S}cience ({B}urlington, {VT}, 1996)}, pages 56--65. IEEE Comput. Soc. Press,
  Los Alamitos, CA, 1996.

\bibitem[Sim88]{DBLP:conf/stoc/STOC20}
Janos Simon, editor.
\newblock {\em Proceedings of the 20th Annual ACM Symposium on Theory of
  Computing, May 2-4, 1988, {Chicago}, {Illinois}, {USA}}. ACM, 1988.

\bibitem[Sop15]{MR3345095}
Ivan Soprunov.
\newblock Lattice polytopes in coding theory.
\newblock {\em J. Algebra Comb. Discrete Struct. Appl.}, 2(2):85--94, 2015.

\bibitem[SS09]{MR2476837}
Ivan Soprunov and Jenya Soprunova.
\newblock Toric surface codes and {M}inkowski length of polygons.
\newblock {\em SIAM J. Discrete Math.}, 23(1):384--400, 2008/09.

\bibitem[Ste96a]{MR1398854}
A.~M. Steane.
\newblock Error correcting codes in quantum theory.
\newblock {\em Phys. Rev. Lett.}, 77(5):793--797, 1996.

\bibitem[Ste96b]{PhysRevA.54.4741}
A.~M. Steane.
\newblock Simple quantum error-correcting codes.
\newblock {\em Phys. Rev. A}, 54:4741--4751, Dec 1996.

\bibitem[Ste96c]{MR1421749}
Andrew Steane.
\newblock Multiple-particle interference and quantum error correction.
\newblock {\em Proc. Roy. Soc. London Ser. A}, 452(1954):2551--2577, 1996.

\bibitem[Ste98]{MR1750540}
Andrew~M. Steane.
\newblock Quantum error correction.
\newblock In {\em Introduction to quantum computation and information}, pages
  184--212. World Sci. Publ., River Edge, NJ, 1998.

\bibitem[Ste99a]{771249}
A.~M. Steane.
\newblock Quantum {R}eed-{M}uller codes.
\newblock {\em IEEE Transactions on Information Theory}, 45(5):1701--1703, Jul
  1999.

\bibitem[Ste99b]{Steane19992492}
A.M. Steane.
\newblock Enlargement of {C}alderbank-{S}hor-{S}teane quantum codes.
\newblock {\em IEEE Transactions on Information Theory}, 45(7):2492--2495,
  1999.

\bibitem[Sti09]{MR2464941}
Henning Stichtenoth.
\newblock {\em Algebraic function fields and codes}, volume 254 of {\em
  Graduate Texts in Mathematics}.
\newblock Springer-Verlag, Berlin, second edition, 2009.

\bibitem[TVZ82]{MR705893}
M.~A. Tsfasman, S.~G. Vl\u{a}du\c{t}, and Th. Zink.
\newblock Modular curves, {S}himura curves, and {G}oppa codes, better than
  {V}arshamov-{G}ilbert bound.
\newblock {\em Math. Nachr.}, 109:21--28, 1982.

\end{thebibliography}

\affiliationone{
   Johan P. Hansen\\
  Department of Mathematics\\
  Aarhus University\\
  Ny Munkegade 118\\
  8000  Aarhus C\\
  Denmark

   \email{matjph@math.au.dk}
}
\end{document}